\documentclass[12pt]{article}
\usepackage{times}
\usepackage[a4paper,text={128mm,185mm}, centering]{geometry}
\usepackage[leqno]{amsmath}%
\usepackage{amsfonts}%
\usepackage{amssymb}%
\usepackage{amsthm}
\usepackage[utf8]{inputenc}
\usepackage[T1]{fontenc}
\usepackage[english]{babel}
\DeclareUnicodeCharacter{0308}{\"}
\usepackage[backend=biber, style=alphabetic, sorting=nyt, url=false, doi=false, isbn=false, eprint=false]{biblatex}
\usepackage[all,cmtip]{xy}%
\usepackage{fancyhdr}
\usepackage{titlesec}
\titleformat{\section}[hang]%
{\bfseries\large}{\thesection.}{1ex}{}%
\titleformat{\subsection}[hang]%
{\bfseries}{\thesubsection}{1ex}{}%
\titleformat{\subsubsection}[runin]%
{\bfseries}{\thesubsubsection}{1ex}{}%
\titleformat{\acknowledgment}[runin]%
{\it}{}{1ex}{}%
\usepackage{csquotes}
\usepackage{graphicx}%
\usepackage{sidenotes}%
\usepackage{xcolor}
\usepackage{newtxtext,newtxmath}%
\usepackage{hyperref}%
\hypersetup{
    colorlinks=true,
    linkcolor=black,
    filecolor=magenta,      
    urlcolor=cyan,
    pdftitle={Normal functor},
    pdfpagemode=FullScreen,
    citecolor=black,
    }
\usepackage{contour} 
\usepackage[normalem]{ulem}

\contourlength{0.8pt}

\usepackage{soul} 
\setul{1pt}{.4pt}
\newcommand{\R}{\mathbb{R}}
\newcommand{\id}{\mathrm{id}}
\newcommand{\flip}{\mathrm{flip}}
\newcommand{\rflip}{\mathrm{r\text{-}flip}}
\newcommand{\sflip}{\mathrm{s\text{-}flip}}
\newcommand{\triv}{\mathrm{triv}}
\newcommand{\coker}{\mathrm{coker}}

\newcommand{\quot}{\mathrm{quot}}
\newcommand{\htov}{{}^v\!\!\!\diagup_{\!\!\! h}}
\newcommand{\varhtov}{{}^v\!\!\!\diagup_{\!\! h}}
\newcommand{\vtoh}{{}^h\!\!\!\diagup_{\!\!\! v}}
\newcommand{\varvtoh}{{}^h\!\!\!\diagup_{\!\! v}}
\newcommand{\range}{\mathrm{Im}}
\newcommand{\pr}{{\mathrm{pr}}}
\newcommand{\Cinf}{C^\infty}
\newcommand{\Hcal}{\mathcal{H}}
\newcommand{\Ecal}{\mathcal{E}}
\newcommand{\Fcal}{\mathcal{F}}
\newcommand{\mathsc}[1]{\normalfont\textsc{#1}} 
\newcommand{\VB}{\mathrm{VB}}
\newcommand{\DVB}{\mathrm{DVB}}
\newcommand{\Smooth}{\mathsc{Smooth}} 
\newcommand{\Embed}{\mathsc{Embed}} 
\newcommand{\Immer}{\mathsc{Immer}} 
\newcommand{\Submer}{\mathsc{Submer}} 
\newcommand{\LieGrpd}{\mathsc{LieGrpd}} 
\newcommand{\Grpd}{\mathsc{Grpd}}
\newcommand{\class}{\mathscr{C}} 
\newcommand{\Cat}{\textsc{Cat}}
\newcommand{\Dscr}{\mathscr{D}}
\newcommand{\dtimes}[2]{{\lceil {#1},{#2} \rceil}}

\newcommand{\vlift}{\mathcal{V}}
\newcommand{\ddtvar}{
        \left.
            \frac{d}{d\lambda}
        \right
        \vert_{\lambda=0}
    }
    
\newcommand{\flipar}{\xymatrix@C=1.2pc{\ar[r] \ar@{}[r]|-{\rule{0.8pt}{5pt}}& }}
\newcommand{\flipisom}{\xymatrix@C=1.4pc{\ar@{<->}[r] \ar@{}[r]|-{\rule{0.8pt}{5pt}}& }}
\newcommand{\flipisomvar}{\xymatrix@C=1.4pc{\ar[r] \ar@{}[r]|-{\rule{0.8pt}{5pt}}& }}
\newcommand{\longflipar}{\xymatrix@C=1.6pc{\ar[r] \ar@{}[r]|-{\rule{0.8pt}{5pt}}&}}
\newcommand{\flipard}{\ar[d] \ar@{}[d]|-{\rule{5pt}{0.8pt}}}
\newcommand{\fliparu}{\ar[u] \ar@{}[u]|-{\rule{5pt}{0.8pt}}}
\newcommand{\fliparr}{\ar[r] \ar@{}[r]|-{\rule{0.8pt}{5pt}}}
\newcommand{\fliparl}{\ar[l] \ar@{}[l]|-{\rule{0.8pt}{5pt}}}
\newcommand{\flipardd}{\ar[dd] \ar@{}[dd]|-{\rule{0.8pt}{5pt}}}
\newcommand{\fliparrr}{\ar[rr] \ar@{}[rr]|-{\rule{0.8pt}{5pt}}}
\newcommand{\fliparuu}{\ar[uu] \ar@{}[uu]|-{\rule{5pt}{0.8pt}}}
\def\rRightarrow{
\ar@{}[rd]|-{\rotatebox[origin=c]{0}{$\Rightarrow$}}}
\def\lRightarrow{
\ar@{}[ld]|-{\rotatebox[origin=c]{0}{$\Leftarrow$}}}
\def\uRightarrow{
\ar@{}[rd]|-{\rotatebox[origin=c]{90}{$\Rightarrow$}}}
\def\dRightarrow{
\ar@{}[rd]|-{\rotatebox[origin=c]{-90}{$\Rightarrow$}}}
\usepackage{trimclip}
\makeatletter
\DeclareRobustCommand{\leftloop}{%
  \mathrel{\mathpalette\left@loop\relax}%
}
\newcommand{\left@loop}[2]{%
  \vphantom{\looparrowright}
  \smash{\clipbox{0 {-.1\height} {.35\width} {-.1\height}}{$\m@th#1{\looparrowright}$}}%
}
\makeatother
\def\loopd{\ar@{}[d]|{\rotatebox[origin=c]{90}{$\looparrowleft$}}}
\def\loopr{\ar@{}[r]|{\rotatebox[origin=c]{0}{$\looparrowleft$}}}
\def\xloopr[#1]{\ar@{}[r]|-{\rotatebox[origin=c]{0}{$\leftloop\joinrel\xrightarrow{\hspace{#1}}$}}}
\def\xloopd[#1]{\ar@{}[d]|-{\rotatebox[origin=c]{-90}{$\leftloop\joinrel\xrightarrow{\hspace{#1}}$}}}
\def\dar[#1]{\ar@<2pt>[#1]\ar@<-2pt>[#1]}
\def\hookd{\ar@{}[d]|{\rotatebox[origin=c]{90}{$\longleftarrow\joinrel\rhook$}}}
\def\xhookd[#1]{\ar@{}[d]|{\rotatebox[origin=c]{90}{$\xleftarrow{\hspace{#1}}\joinrel\rhook$}}}
\def\hookr{\ar@{}[r]|-{\rotatebox[origin=c]{0}{$\lhook\joinrel\longrightarrow$}}}
\def\xhookr[#1]{\ar@{}[r]|-{\rotatebox[origin=c]{0}{$\lhook\joinrel\xrightarrow{\hspace{#1}}$}}}
\def\xhooku[#1]{\ar@{}[u]|-{\rotatebox[origin=c]{90}{$\lhook\joinrel\xrightarrow{\hspace{#1}}$}}}
\theoremstyle{plain}
\newtheorem{theorem}{Theorem}[section]
\newtheorem{prop}[theorem]{Proposition}
\newtheorem{cor}[theorem]{Corollary}

\newtheorem{lem}[theorem]{Lemma}
\theoremstyle{definition}

\newtheorem{deff}[theorem]{Definition}
\theoremstyle{remark}

\newtheorem{remark}[theorem]{Remark}

\newtheorem{example}[theorem]{Example}
\makeatletter
\newcommand{\leqnomode}{\tagsleft@true}
\newcommand{\reqnomode}{\tagsleft@false}
\makeatother
\reqnomode
\title{\vskip 5pt  \bf ON THE NORMAL FUNCTOR IN THE CATEGORY OF SMOOTH VECTOR BUNDLES}
\author{\itshape\bfseries{Quentin Karegar Baneh Kohal}}
\date{}
\begin{document}
\maketitle

\thispagestyle{empty}
\vskip 25pt
\begin{adjustwidth}{0.5cm}{0.5cm}
{\small
{\bf R\'esum\'e.} Cet article est d\'edi\'e \`a l'\'etude du foncteur normal dans la cat\'egorie des fibr\'es vectoriels r\'eels lisses. 
Plus particuli\`erement, nous \'etudions un ph\'enom\`ene de sym\'etrie lors de l'it\'eration double du foncteur normal sur un carr\'e commutatif d'immersions lisses. 
Pour ce faire, une th\'eorie de tir\'e en arri\`ere et de quotient est d\'evelopp\'e pour les fibr\'es vectoriels doubles mais \'egalement pour certaines classes de morphismes. 
Ces deux op\'erations r\'esultent \^etre les ingr\'edients n\'ecessaires pour \'etudier la naturalit\'e du foncteur normal. 
Ainsi, la sym\'etrie attendue est obtenue gr\^ace 
au caract\`ere universel et \`a la compatibilit\'e mutuelle de ces op\'erations.
\\
{\bf Abstract.} This article is dedicated to the study of the normal functor in the category of smooth real vector bundles. 
Particularly, we focus on a symmetry phenomena which occurs after iterating two times the normal functor on a commutative square of smooth immersions.
To do so, a theory of pullback and quotient is developed for double vector bundles but also for some classes of morphisms. These two operations turn out to be the key ingredients in order to study the naturality of the normal functor. 
The expected symmetry is then obtained thanks to the universal behavior and the mutual compatibility of these operations.
}
\end{adjustwidth}
%
%
%
%
%
\section*{Introduction}
\paragraph{}
Nowadays, it is well-known that the first order approximation of the geometry  along a submanifold of a smooth manifold is modeled  by its normal bundle. 
In fact, the normal bundle makes sense for any immersion of smooth manifolds and enjoys some functorial properties with respect to ``morphisms between immersions'', which are given by some special kind of commutative squares generalizing in a naive way the inclusion of pairs of manifolds.
At this point, a preliminary issue arises: given a morphism of immersions, when does the induced map between the respective normal bundles is itself an ``immersion''?
Recall that this latter map, called the \textit{normal differential}, is first of all a vector bundle morphism, hence the word ``immersion'' here stands for a suitable adaptation of the notion of immersion of smooth manifolds to the category of vector bundles. 
In order to obtain a positive answer to the previous question,  we are forced to consider not only morphism of immersions, but ``immersions of immersions'' (commutative squares consisting of immersions). 
The only small dampler is that even under this latter restriction, there is no guarantee to obtain an immersion of vector bundle after applying the normal functor.
However that issue get easily solved, but not without revealing a new ambiguity: the symmetry of the situation, namely immersions of immersions have two distinct read directions.
Assuming that both induce an immersion of vector bundles, our main concern is to determine if the different ways of iterating (two times) the normal bundle return the same object, named by double normal bundle'. More generally, we aim for such result in the case of square of Lie groupoid immersions.
\par \hspace{1pc}
The content of the present work consists in bringing such questionning in the shed of (strict) double categories as introduced by  Charles Ehresmann in~\cite{Ehresmann1963double}. 
 A simple example of such structure is, suggestively, given by a (single) double vector bundles~\cite{pradines1974dvb-cras}, which arises naturally when applying the normal functor onto an immersion of vector bundle.
In a slightly different direction, the class of commutative squares of smooth maps  define as well a double category, denoted $\Smooth^\square$. 
In fact, at some point we should consider a mixture of the examples just mentionned, culminating in the double category $\DVB^\square$ whose objects consists in commutative squares of double vector bundle.
%
More specifically, the construction of the normal bundle rests on two operations: the pullback (more generally, the fiber products) and the quotient.
These latter operations will not be limited to objects only, but be also considered for structured objects (vector bundles, double vector bundles) as well as for morphisms.
\par \hspace{1pc}
%
%
After defining the relevant double categories, we recall some preliminary properties of the normal functor in the category of smooth manifolds (section~\ref{sec.preliminaries}).
In particular, we introduce the fundamental notion of pullback of vector bundle morphism by (directed) double morphism. The vertical functoriality is delayed until appendix~\ref{app.vert-fun}, due to its conceptual nature.
%
%
Next, we focus on double vector bundles (section~\ref{sec.dvb}) by providing a rudimentary exposition of some generalized pullback, called \textit{side-pullback}, which get declined into two types: vertical and horizontal 
Then, we introduce some refinement of the category $\DVB$ of double vector bundles in order to realize quotients as genuine cokernels in some specific categories, again either horizontal or vertical.  
The normal bundle of an immersion of vector bundle is finally obtained by mimicking the usual definition for immersion of smooth manifolds: pulling back and quotienting.
%
%
The section~\ref{sec.symm-thm} is dedicated to the proof of the ``symmetry theorem'' (theorem~\ref{thm.main}) which asserts that, given a suitable square of smooth immersions, iterating the normal bundle does not depend on the direction chosen (up to some natural flip). 
The particular case of the symmetry theorem for embedding of smooth manifolds has been studied in~\cite{pikethesis}, and in~\cite{meinrenken2022dvbquot} with a different perspective in mind (quotient of multigraded bundles) from ours, and using techniques which do not seem affordable in the context of Lie groupoids.
Our strategy is to set up a convenient framework in order to exhibit the desired isomorphism as a double quotient in the category of flip isomorphism. 
%
%
%
%
The basic underlying ideas of double category theory are compiled in the appendix~\ref{app.dbcat}, particularly a precise notion of direction in a double category (horizontal/vertical) which are interrelated through a ``flip'' of a double category (essentially the notion of transpose in \cite{grandis-pare1999}).
\par \hspace{1pc}
All along the paper, we juggle between the classical definition of double vector bundle, as presented in~\cite{mackenzie2005}, and the one using homogeneity structures \cite{grabowski-rotkiewicz2009}.
This latter viewpoint vastly simplifies the transition from smooth manifolds to Lie groupoids, thanks to~\cite{bursztyn-cabrera-delhoyo-2016}.
The upside of our method is to generalize directly to the Lie groupoid framework, resulting in a version of the symmetry theorem for double vb-groupoids (section~\ref{sec.normalgroupoid}). 
Notice that such structure seems to have made only sparse apparitions in the literature. For instance, as a special case of multiple graded groupoids from~\cite{bruce-grabowska-grabowski2015graded}, for which a suitable version of the symmetry theorem would arguably hold.
Our main motivation comes from their natural occurrence in index theory through the study of double deformations (to the -double- normal groupoid) of Lie groupoids, which will be developed in a subsequent work in collaboration with Paulo Carrillo Rouse. 
%
%
%
\paragraph{Conventions and notations.} 
All manifolds, maps and vector bundles are supposed to be over $\R$ and of class $\Cinf$ (smooth), unless specified.
Vector sub-bundles are not supposed to be wide, that is the base manifolds do not necesarily coincide, moreover we tacitly assume that vector bundles are of constant rank (even when the base is not connected).
Vector bundles morphisms are shortened into ``vb-maps'' and are often summoned as pair of maps $(\varphi, f)$, by precising the base map $f$.
The word "embedding" will stand for injective immersion, whereas "proper embedding" will refer to a proper injective immersion.
We will use the arrows $\hookrightarrow$ and $\looparrowright$ to denote embeddings and immersions, respectively.
If $f_1: M \rightarrow N$ and $f_2: M \rightarrow P$ are two smooth maps, then $\dtimes{f_1}{f_2} := (f_1 \times f_2)\circ \Delta_{M}$ define a map $M \rightarrow N \times P$ where $\Delta_M: M \hookrightarrow M \times M$ is the diagonal embedding.
%
The differential of a smooth map $f: M \rightarrow N$ of smooth manifolds is denoted by $f_*: TM \rightarrow TN$, and its restriction to a fiber is $f_{*,m}: T_m M \rightarrow T_{f(m)}N$. 
In the particular case of a smooth action $\alpha: \R \times M \rightarrow M$ of the monoid $(\R, \cdot)$, we use the notation $\alpha_*$ to refer to the induced action on the tangent bundle $\alpha_*: \R \times TM \rightarrow TM$, called the tangent action.
Throughout the text, we use a slight modification of the usual differential map $f_\sharp= \dtimes{\pi}{f_*}: TM \rightarrow f^*TN$, called the sharp differential, in particular $f_\sharp $ is a morphism of vector bundles over $\id_M$. More generally, given a vector bundle map $\varphi: E \rightarrow F$ over $f: M \rightarrow N$, the sharpening of $\varphi$ will be the vb-map $\varphi_\dag$ defined as $\dtimes{\pi}{\varphi}: E \rightarrow f^*F \subseteq M \times F$. An identification (or a 2-isomorphism) between two morphisms $f$ and $g$ in a given category $C$ is a pair of isomorphism $(\psi,\varphi)$ in $C$ such that $f \psi = \varphi g$. When the isomorphisms $\psi$ and $\varphi$ are natural, $f$ and $g$ are said to be ``naturally identified'' or ``essentially equal''. 
\paragraph{Acknowledgement.}
The author want to thank Paulo Carrillo Rouse and Jos\'e Luis Cisneros Molina for valuable discussion. 
This work was supported by 
DGAPA-Universidad Nacional Aut\'onoma de M\'exico through Proyecto PAPIIT IN105121.
%
%
%
\section{Normal functor: construction}\label{sec.preliminaries}
\subsection{Some double categories.}
\subsubsection{}
Set $\Smooth$ to be the category of smooth manifolds and smooth maps.
Let $\class_1, \class_2$ be two wide subcategories of $\Smooth$, each of which has morphisms consisting of a -eventually distinct- class of smooth maps stable under base-change (e.g. embeddings, immersions, submersions, surmersions, vector bundle projections, proper maps, etc...). 
Consider the double category\footnote{We only consider strict double categories as in the original source~\cite{Ehresmann1963double}, cf.~appendix~\ref{app.dbcat} for further details.}  
$(\class_1, \class_2)^\square$ described by:
\begin{itemize}
    \setlength\itemsep{0pc}
    \item Objects: smooth manifolds;
    \item Vertical morphisms: morphisms in $\class_1$;
    \item Horizontal morphisms: morphisms in $\class_2$;
    \item Double morphisms: commutative squares (whose horizontal/vertical edges are horizontal/vertical morphisms).
\end{itemize}
Given a wide subcategory $\class$ as above, we introduce the double categories 
$$
    \class^v = (\class, \Smooth)^\square, \quad \class^h = (\Smooth, \class)^\square, \quad \class^\square = (\class,\class)^\square.
$$ 
For instance, if $\class= \Immer$ denotes the category with smooth manifolds as objects and immersions as morphisms, then the previously defined double categories fit into a comparison table as follows:
\begin{center}
\begin{tabular}{|c|c|c|c|c|}
    \hline
    Double category 
    & $\Immer^v$ 
    & $\Immer^h$ 
    & $\Immer^\square$
    \\ \hline
    Objects 
    & manifolds
    & manifolds
    & manifolds
    \\  \hline
    Vertical morphisms 
    & immersions
    & smooth maps
    & immersions
    \\  \hline
    Horizontal morphisms 
    & smooth maps 
    & immersions
    & immersions
    \\  \hline
    Double morphisms 
    & com. squares
    & com. squares
    & com. squares
     \\  \hline
\end{tabular}
\end{center} 
\subsubsection{The double category $\Immer$.} 
Let $j_M : M_s \looparrowright M_r$ and $j_N: N_s \looparrowright N_r$ be two immersions. A \textbf{2-map of immersions}, denoted $F: j_M \Rightarrow j_N$, consists of a pair of smooth maps $f_s: M_s \rightarrow N_s$ and $f_r: M_r \rightarrow N_r$ making the following diagram commute:
\begin{align}\label{morphism-embed} \begin{aligned}
 \xymatrix@C=1.5pc@R=1.2pc{
M_s \ar[r]^-{f_s} 
\xloopd[0.7pc]_-{j_M}
\rRightarrow
& N_s
\xloopd[0.7pc]^-{j_N}
\\
M_r 
\ar[r]_-{f_r}  
& N_r
}   
\end{aligned}
\end{align}
 Notice that, as formulated, the latter square is \textit{directed} in the sense that it is interpreted as a 2-maps in the \textit{horizontalization} of the double category $\Immer^v$  (cf.~\S\ref{par.choosing-direction}). 
 In a bold move, such horizontalization will be still denoted as $\Immer$, in such a way that a 2-map of immersions is the same as a 2-map in $\Immer$.
 The horizontal composition $\circ$ and vertical compositions $\bullet$ are respectively given by horizontal and vertical concatenations of squares. 
 Let us specify the conventions used for the order of composition: 
 \begin{itemize}
  \item if $F= (f_2,f_1)$ and $G= (g_2,g_1)$ are horizontally composable, then their composition is $G\circ F = (g_2 \circ f_2, g_1 \circ f_1)$:
 $$
 \left(
 \vcenter{
  \xymatrix@C=1.5pc@R=1.2pc{
    P_1 
    \xloopd[0.7pc]_-{k}
    &\ar[l]_-{g_1}  N_1
    \lRightarrow
    \xloopd[0.7pc]^-{j}
    \\
    P_2 
    & \ar[l]^-{g_2}N_2
 }}\right)
 \:\circ\: 
 \left(\vcenter{
  \xymatrix@C=1.5pc@R=1.2pc{
    N_1  
    \xloopd[0.7pc]_-{j}
    & \ar[l]_-{f_1} M_1
    \xloopd[0.7pc]^-{i}
    \lRightarrow
    \\
    N_2 
    & \ar[l]^-{f_2}M_2
 }}\right) 
 \quad = \quad 
 \left(\vcenter{
  \xymatrix@C=1.8pc@R=1.2pc{
    P_1 \ar[r]^-{g_1 \circ f_1} 
    \xloopd[0.7pc]_-{k}
    & N_s
    \xloopd[0.7pc]^-{j}
    \lRightarrow
    \\
    P_2 
    \ar[r]_-{g_2 \circ f_2}  
    & N_r
 }}\right)
 $$
 \item if $F = (f_2,f_1)$ and $H =(h_2,h_1)$ are vertically composable, that is $f_2 = h_1$, then their composition $H \bullet F =(h_2,f_1)$ yields a 2-map $i_2 \circ i_1 \Rightarrow j_2 \circ j_1$ in $\Immer$: 
 $$
 \left(\vcenter{
  \xymatrix@C=1.5pc@R=1.2pc{
    N_2  
    \xloopd[0.7pc]_-{j_2}
    & \ar[l]_-{h_1} M_2
    \xloopd[0.7pc]^-{i_2}
    \lRightarrow
    \\
    N_3 
    & \ar[l]^-{h_2}M_3
 }}\right)  
 \: \bullet \: 
 \left(\vcenter{
  \xymatrix@C=1.5pc@R=1.2pc{
    N_1  
    \xloopd[0.7pc]_-{j_1}
    & \ar[l]_-{f_1} M_1
    \xloopd[0.7pc]^-{i_1}
    \lRightarrow
    \\
    N_2 
    & \ar[l]^-{f_2}M_2
 }}\right)  
 \quad = \quad 
 \left(\vcenter{
  \xymatrix@C=1.5pc@R=1.2pc{
    N_1  
    \xloopd[0.7pc]_-{j_2 \circ j_1}
    & \ar[l]_-{f_1} M_1
    \xloopd[0.7pc]^-{i_2 \circ i_1}
    \lRightarrow
    \\
    N_3 
    & \ar[l]^-{h_2}M_3
 }}\right) 
 $$
 \end{itemize}
For the sake of compactness, the square~(\ref{morphism-embed}) will sometimes be written as $$F=(f_r,f_s): (j_M: M_s \looparrowright M_r) \Rightarrow (j_N: N_s \looparrowright N_r).$$
\begin{remark}
  Replacing the immersions by another wide sub-category $\class$ leads to the notion of 2-maps of $\class$-maps. As an example, if $\class= \Submer$ is the category of smooth manifolds with the submersions as morphisms, then a commutative square in which the vertical arrows are submersions and the horizontal arrows are smooth maps defines a 2-map of submersions.
\end{remark}

\begin{remark}\label{rem.flip}
    The double morphism represented by the commutative square~(\ref{morphism-embed}) has another interpretation: the 2-morphism $J: f_s \Rightarrow f_r$ in the \textit{verticalization} of $\Immer^v$ (left-hand side below), which clearly deserves the name of \textbf{2-immersion} (of smooth maps).
    Alternatively but equivalently, $J: f_s \Rightarrow f_r$ may be thought as a 2-morphism in the \textit{horizontalization} of the double category $\Immer^h$ by ``flipping the square'' (right-hand side below). See~\S\ref{par.choosing-direction} and~\S\ref{par.dbcat-flip}.
    \begin{align} 
    \begin{split}
   \xymatrix@C=1.5pc@R=1.2pc{
    M_s \ar[r]^-{f_s} 
    \xloopd[0.7pc]_-{j_M}
    \dRightarrow
    & N_s
    \xloopd[0.7pc]^-{j_N}
    \\
    M_r 
    \ar[r]_-{f_r}  
    & N_r
    }   
    \qquad \quad 
    \xymatrix@C=1.5pc@R=1.2pc{
    M_s 
    \ar[d]_-{f_s} 
    \xloopr[0.8pc]^-{j_M}
    \rRightarrow
    & M_r 
    \ar[d]^-{f_r}
    \\
     N_s
      \xloopr[0.8pc]_-{j_N}
    & N_r
    }   
\end{split}
\end{align}
\end{remark}
\begin{example}\label{ex.map-as-2map}
A smooth map $f: M \rightarrow N$ may be lifted into an horizontal 2-map of $\Smooth^\square$ in at least four different ways: either $\id_M \Rightarrow f$ (the range lift), $f \Rightarrow \id_N$ (the source lift), $\id_M \Rightarrow \id_N$ (the diagonal lift), or $f \rightarrow f$ (the trivial lift) which are respectively given by the following commutative squares
$$
\xymatrix@C=1.5pc@R=1.5pc{
 M \ar[r]^-{\id} \ar[d]_-{\id} 
  \ar@{}[rd]|-{\Rightarrow}
 & M \ar[d]^-{f}
 \\
 M \ar[r]_-{f} &  N
} 
\qquad 
\xymatrix@C=1.5pc@R=1.5pc{
 M \ar[r]^-{f} \ar[d]_-{f} 
 \ar@{}[rd]|-{\Rightarrow}
 & N \ar[d]^-{\id}
 \\
 N \ar[r]_-{\id} &  N
}
\qquad 
\xymatrix@C=1.5pc@R=1.5pc{
 M \ar[r]^-{f} \ar[d]_-{\id} 
 \ar@{}[rd]|-{\Rightarrow}
 & N \ar[d]^-{\id}
 \\
 M \ar[r]_-{f} &  N
}
\qquad 
\xymatrix@C=1.5pc@R=1.5pc{
 M \ar[r]^-{\id} \ar[d]_-{f} 
 \ar@{}[rd]|-{\Rightarrow}
 & M \ar[d]^-{f}
 \\
 N \ar[r]_-{\id} &  N
}
$$
\end{example}

\begin{example}\label{ex.vbmap-as-2map}
    Let $\pi_1: E_1 \rightarrow M_1$ and $\pi_2: E_2\rightarrow M_2$ be two vector bundles and $\varphi:E_1 \rightarrow E_2$ be a vb-map, covering $f: M_1 \rightarrow M_2$. 
    Then the pair $(\varphi,f)$ defines a 2-map of embeddings $0_1 \Rightarrow 0_2$ between the respective zero section, but also a 2-map of vector bundle projections $(f,\varphi): \pi_1 \Rightarrow \pi_2$ (in particular, a 2-map of surmersions).
\end{example}
%
%
\subsubsection{The double category $\Immer\VB$.}
Let $\VB$ be the category of vector bundles and vb-maps. There is an obvious functor $\pi: \VB \rightarrow \Smooth$ which assigns to each vector bundle (resp. vb-map) its base manifold (resp. base map). 
As before, we consider the double category $\VB^\square$ whose double morphisms consist in commutative squares of vb-maps, which are translated into  horizontal or vertical \textbf{2-vb-maps} after choosing a direction. 
In particular, the base-projection functor $\pi$ extends into a double functor $\pi^\square: \VB^\square \rightarrow \Smooth^\square$, that projects a 2-vb-map to its base 2-map within the horizontalization and verticalization.
\begin{deff}\label{def.vbimmersion}
    An \textbf{immersion of vector bundle} (shortly, vb-immersion) is a vb-map which is fiberwise injective and whose base map is an immersion of smooth manifolds. The double category $\Immer\VB^v$ is the double subcategory of $\VB^\square$ with vb-immersions as vertical morphisms. As before, we denote by $\Immer\VB$ its horizontalization.
    The (horizontal) 2-morphisms in $\Immer\VB$ are called \textbf{2-maps of vb-immersions}.
\end{deff}
Clearly, the double functor $\pi^\square$ induces a double functor $\Immer\VB \rightarrow \Immer$. 
The verticalization of a (horizontal) 2-map of vb-immersions will be called \textbf{2-vb-immersion} (of vb-maps), following the terminology of remark~\ref{rem.flip}.
\begin{remark}
Equivalently, a vb-immersion is vb-map which realizes an immersion between the total spaces. The analogous definitions give the notions of vb-submersion, vb-embedding, and so on... 
\end{remark}

\begin{example}
  Applying the tangent functor to a 2-map of immersions $F= (f_2,f_1): j_1 \Rightarrow j_2$ yields the 2-map of vb-immersions $F_*= (f_{2*},f_{1*}): j_{1*} \Rightarrow j_{2*}$. Going backward, the evaluation of the double functor $\pi^\square$ on $F_*$ recovers $F$.
\end{example}
\subsection{Basic constructions}
The present section is devoted to introduce some notations regarding a certain amount of constructions: fiber products of vb-maps, quotient of vb-maps, pullback of vb-maps by 2-maps. We also review the definition of normal bundle attached to an immersion together with some observations about its functorial behavior within the double category $\Immer$.
\subsubsection{Fiber product of vb-maps.}
\begin{deff}
 A cospan $(E_1 \xrightarrow{\varphi_1} E \xleftarrow{\varphi_2} E_2)$ in $\VB$, with base cospan $(M_1 \xrightarrow{f_1} M \xleftarrow{f_2} M_2)$ in $\Smooth$, is called \textbf{good} if 
 \begin{itemize}
  \setlength\itemsep{0pc}
  \item[(i)] $M_1 \times_{f_1,f_2} M_2$ is an embedded smooth submanifold of $M_1 \times M_2$ such that $T(M_1 \times_{f_1,f_2} M_2)$ is isomorphic to  $TM_1 \times_{f_{1*},f_{2*}} TM_2$ in a natural way;
  \item[(ii)] $E_1 \times_{\varphi_1,\varphi_2} E_2 \rightarrow M_1 \times_{f_1,f_2} M_2$ is a smooth vector sub-bundles of the product $E_1 \times E_2 \rightarrow M_1 \times M_2$ in a natural way.
 \end{itemize}
 \end{deff}
 \begin{remark}
 The previous definition implies that both of the following sequences of vector bundles are exact:
 \begin{itemize}
 \setlength\itemsep{0pc}
  \item[]  $
 \xymatrix{
    0 \ar[r] & T(M_1 \times_M M_2) \ar[r] & (TM_1 \times TM_2)\vert_{M_1 \times_M M_2} \ar[r]^-{f_{1*}-f_{2*}} 
    & TM,
 }
 $
 \item[]  $
 \xymatrix{
     0 \ar[r]& E_1 \times_E E_2 \ar[r] & (E_1 \times E_2)\vert_{M_1 \times_M M_2} \ar[r]^-{\varphi_1-\varphi_2} 
    & E.
 }
 $
 \end{itemize}
\end{remark}
Let $(E_1 \xrightarrow{\varphi_1} E \xleftarrow{\varphi_2} E_2)$  be a good cospan in $\VB$ with base cospan $(M_1 \xrightarrow{f_1} M \xleftarrow{f_2} M_2)$.
Then,  there is a induced vector bundle structure\footnote{In fact, it carries a richer structure of ``fiber product object'' in $\VB$.} $\pi_1 \times_\pi \pi_2: E_1 \times_{E} E_2 \rightarrow M_1 \times_M M_2$, called the \textbf{fiber product} of $E_1$ and $E_2$ along $E$ in the category $\VB$.
The fiber product vector bundle, when it exists, fits into a commutative square 
$$
\left.
    \vcenter{
\xymatrix@C=1.5pc@R=1.5pc{
   E_1 \times_E E_2 
   \ar[r]^-{\pr_1} 
   \ar[d]_-{\pr_2}
   & E_1  \ar[d]
   \\
   E_2 
   \ar[r]
   & E 
}}
\right.
\quad \text{ with base square }
\left.
    \vcenter{
\xymatrix@C=1.5pc@R=1.5pc{
   M_1 \times_M M_2 
   \ar[r]^-{\pr_1} 
   \ar[d]_-{\pr_2}
   & M_1  \ar[d]
   \\
   M_2 
   \ar[r]
   & M 
}}
\right.
$$
Let $\Ecal: (E_1 \xrightarrow{\varphi_1} E \xleftarrow{\varphi_2} E_2)$ and 
$\Fcal: (F_1 \xrightarrow{\psi_1} F \xleftarrow{\psi_2} F_2)$ be two cospan in $\VB$ with base cospan $(M_1 \xrightarrow{f_1} M \xleftarrow{f_2} M_2)$ and $(N_1 \xrightarrow{g_1} N \xleftarrow{g_2} N_2)$ respectively.
Assume the existence of a morphism of cospan $\Ecal \rightarrow \Fcal$, that is a tuple $(\theta_1,\theta_2,\theta)$ of vb-maps $\theta_1: E_1 \rightarrow F_1$, $\theta_2: E_2 \rightarrow F_2$, and  $\theta: E \rightarrow F$ compatible with the cospan structure. 
Then, by universality, the product vb-map $\theta_1 \times \theta_2$ induces a vb-map $\theta_1 \times_\theta \theta_2$ between the corresponding fiber product vector bundles\footnote{Namely, a morphism of ``fiber product object'' in $\VB$.}:
$$
\xymatrix@C=1.8pc@R=1.5pc{
   E_1 \times_E E_2 
   \ar[r]^-{\theta_1 \times_\theta \theta_2} 
   \ar[d]
   & F_1 \times_F F_2  \ar[d]
   \\
   M_1 \times_M M_2 
   \ar[r]^-{f_1 \times_f f_2}
   & N_1 \times_N  N_2 
}
$$
called the \textbf{fibre product map} of $\theta_1$ and $\theta_2$ along $\theta$.
The construction is clearly compatible with the base-projection functor $\pi:\VB \rightarrow \Smooth$. In other words, the fiber product operation provides a natural transformation from good cospan in $\VB$ to $\VB$.
%
%
\subsubsection{Quotient of vb-maps.}\label{par.quotvbmap}
Let $j_\ell: F_\ell \hookrightarrow E_\ell$ be a wide vb-embedding for $\ell=1,2$, and let $(\varphi,\psi): j_1 \Rightarrow j_2$ be a 2-map in $\Embed\VB$, in other words $\varphi: E_1 \rightarrow E_2$ and $\psi: F_1 \rightarrow F_2$ are vb-maps such that $\varphi \circ j_1 = j_2 \circ \psi$. 
Then, the \textbf{quotient vb-map} $\varphi/\psi$ is the unique map such that the following diagram is a morphism of short exact sequences:
$$
    \xymatrix@R=1.5pc@C=2pc{
       0 \ar[r]
       & F_1 \ar[r]^-{j_1} \ar[d]_-{\psi}
       & E_1 \ar[r]^-{q_1} \ar[d]^-{\varphi}
       & E_1 / F_1 \ar[r] \ar[d]^-{\varphi / \psi}
       & 0
       \\
       0 \ar[r]
       & F_2 \ar[r]_-{j_2} 
       & E_2 \ar[r]_-{q_2} 
       & E_2 / F_2 \ar[r]
       & 0
    }
$$
where $q_1,q_2$ are the quotient projections.
Explicitly, $(\varphi/\psi)([e]_{F_1}) = [\varphi(e)]_{F_2}$ where $[e]_{F_1}$ is the equivalence class\footnote{Notice that we slightly abuse of the notation by denoting $E_\ell / F_\ell$ instead of $E_\ell / j_\ell(F_\ell)$, in particular $\varphi/\psi$ depends also on the vb-emdeddings $j_1$ and $j_2$.} of an element $e \in E_1$ inside $E_1/F_1$, and similarily $[-]_{F_2}$ denotes a class in $E_2/F_2$. 
\begin{remark}
    If the vb-embeddings $j_1, j_2$ are not wide, then one can still replace them by their sharpening $j_{1 \dag},\, j_{2\dag}$ in order to get into the above situation. Moreover, the quotient vb-map is well defined in the broader situation of constant-rank vb-maps (instead of vb-embedddings).
\end{remark}
\begin{remark}\label{rem.hor-cokernel} 
 Let $\Hcal_f$ be the subcategory of morphisms in the horizontalization of the double category $\VB^\square$ (cf.~\S \ref{par.choosing-direction}) consisting of squares with horizontal arrows being wide vb-maps as morphisms (the base map is the identity map), and vb-maps over a fixed smooth map $f: M\rightarrow N$ as objects. 
 Notice that the horizontal concatenation makes sense in $\Hcal_f$ but the vertical one does not (unless $f$ has same source and range). 
 From this perspective, the map $f$ is identified to the zero object in $\Hcal_f$, and  $(q_2,q_1) : \varphi \Rightarrow \varphi/\psi$ is a cokernel for $J=(j_2,j_1)$.
\end{remark}
%
%
%
\subsubsection{Pullback vb-maps by 2-maps.}\label{par.pb-vbmap}
Let $\pi_1: E_1 \rightarrow M_1$ and $\pi_2: E_2 \rightarrow M_2$ be two vector bundles, and $\varphi: E_1 \rightarrow E_2$ be a vb-map with base map $f: M_1 \rightarrow M_2$.
Given a smooth map $g: P_1 \rightarrow P_2$ together with a a 2-map $H = (h_2, h_1): g \Rightarrow f$, the \textbf{pullback} of the vb-map $(\varphi,f)$ along $H$ is defined as the vb-map 
$$
\xymatrix@C=3pc@R=1.5pc{
     h_1^* E_1 
     \ar[r]^-{H^*\varphi}\ar[d] 
     & h_2^*E_2 \ar[d]
     \\
     P_1 \ar[r]^-{g} & P_2
}
$$
given by the fiber product vb-map
$
    H^* \varphi= g \times_f \varphi : P_1 \times_{h_1,\pi_1} E_1 \rightarrow P_2 \times_{h_2,\pi_2} E_2$ associated to the cospan $(H,\underline{ \smash \pi}):(g \Rightarrow f \Leftarrow \varphi)$ of 2-map in $\Smooth$, where $\underline{\smash  \pi}$ is the vb-submersion $(\pi_2,\pi_1)$.
%
%
\par
Furthermore, given a pair of horizontally composable 2-maps $l \overset{K}{\Rightarrow} g\overset{H}{\Rightarrow} f$ and a vb-map $\varphi$ over $f$, we have a natural identification\footnote{Notice that, strictly speaking, $(K\circ H)^*$ and $H^* K^*$ aren't functors since the collection of vb-maps over a fixed smooth map $f$ does not form a category, in general.}
\begin{align}
 (H \circ K)^*\varphi \cong K^*(H^* \varphi)
\end{align}
induced by the natural isomorphisms $(h_\ell \circ k_\ell)^* E_\ell \cong k_\ell^*h_\ell^* E_\ell$ with $\ell = 1,2$.
\par
Regarding the vertical counterpart, let $H = (h_2, h_1)$ and $K = (k_2,k_1)$ be a vertically composable pair of 2-maps, that is $k_1 = h_2$, then 
\begin{align}
 (K \bullet H)^* (\psi \circ \varphi) = (K^*\psi) \circ (H^*\varphi).
\end{align}
\begin{example}[Sharpening of 2-vb-maps.]\label{ex.sharp-2vbmap}
If $\Psi = (\psi_2,\psi_1): \varphi_1 \Rightarrow \varphi_2$ is a 2-vb-map with base 2-map $G= (g_2,g_1): f_1 \Rightarrow f_2$, then $G^*\varphi_2$ is a vb-map over $f_1$. 
We define the \textbf{sharpening} of $\Psi$ to be the 2-vb-map ${\Psi}_\dag: \varphi_1 \Rightarrow G^*\varphi_2$,  over the identity 2-map $ f_1 \Rightarrow f_1$,
given by ${\Psi}_\dag = (\psi_{2\dag},\psi_{1\dag})$ where $\psi_{\ell \dag} = \dtimes{\pi}{\psi_\ell}$.
In particular, the following triangle of vb-maps and 2-vb-maps between them commutes:
$$
\xymatrix@C=1pc@R=2pc{
    \varphi_1 \ar@{=>}[rr]^-{{\Psi}_\dag}  \ar@{=>}[rd]_-{\Psi}
    && G^*\varphi_2 \ar@{=>}[ld]^{\pr}
    \\
    & 
    \varphi_2
    &}
$$
where $\pr$ is the obvious projection.
\end{example}
\begin{example}\label{ex.sharp-diff}
  The sharp differential $f_\sharp: TM \rightarrow f^*TN$ of a smooth map $f: M \rightarrow N$ coincides, up to the tautological natural isomorphism $\id^*TM \cong TM$, with the pullback $(\underline{\smash f})^* f_*: \id^*TM \rightarrow f^*TN$ of $f_*: TM \rightarrow TN$ along the range lift $\underline{\smash f}: \id_M \Rightarrow f$ introduced in example~\ref{ex.map-as-2map}.
\end{example}
\begin{example}\label{ex.unbalanced}
 Let $F=(f_2, f_1)$  be a 2-map of immersions\footnote{This can be adapted for an arbitrary square of smooth maps, nevertheless one has to precise the direction of the pullback (horizontal/vertical).} $j_1 \Rightarrow j_2$ . 
 The tangent map $j_{2*}$ is a vb-immersion, hence its pullback $F^* j_{2*}$ also is. Thus,  $F_\sharp = (f_{2\sharp}, f_{1\sharp})$  defines a 2-map $j_{1*} \Rightarrow F^* j_{2*}$ in $\Immer\VB$, called the \textbf{sharp differential} of $F$.
 The verticalization $f_{1\sharp} \Rightarrow f_{2\sharp}$ of the square representing $F_\sharp$ (see below) produces a 2-vb-immersion $J_\flat = (F^*j_{2*}, j_{1*})$, called the \textbf{flat differential} of the 2-immersion $J= (j_2, j_1): f_1 \Rightarrow f_2$. 
 The double morphism $(F_\sharp, J_\flat)$ in $\VB^\square$ will be called the \textbf{h-sharp differential} of the double morphism $(F,J)$ in $\Smooth$. The h-sharp differential is clearly compatible with the horizontalization/verticalization.
\begin{align}\label{diag.unbalanced}
\begin{split}
     \xymatrix{
            TM_1 \ar[r]^-{f_{1\sharp}} \ar[d]_-{j_{1*}} \ar@{}[rd]|-{(F_\sharp, J_\flat)}& 
        f_1^*TM_2 \ar[d]^-{F^*j_{2*}} 
        \\
        TN_1 \ar[r]_-{f_{2\sharp}} &
        f_2^*TN_2
    }
    \end{split}
\end{align}
Notice that there is an alternative given by the 2-vb-square $(F_\flat, J_\sharp)$, say the ``v-sharp differential''.
\end{example}
%
%
%
%
\subsubsection{Normal bundle (immersion case).}
Let $F = (f_2,f_1) : j_1 \Rightarrow j_2 $ be a 2-map in $\Immer$, then the verticalization $J_\sharp$ of the unbalanced differential fits into the following morphism of exact sequences in $\VB$, depicted (after flipping) as:
$$
    \xymatrix@R=1.5pc@C=3pc{
       0 \ar[r]
       & TM_1 \ar[r]^-{j_{1\sharp}} \ar[d]_-{f_{1*}} 
       \ar@{}[rd]|-{\rotatebox[origin=c]{0}{$\overset{J_\sharp}{\Rightarrow}$}}
       & j_1^*TN_1 \ar[r] \ar[d]^-{J^* f_{2*}}
       \rRightarrow
       & \nuup(j_1) \ar[r] \ar[d]^-{J^* f_{2*} / f_{1*}}
       & 0
       \\
       0 \ar[r]
       & TM_2 \ar[r]_-{j_{2\sharp}} 
       & j_2^*TN_2 \ar[r] 
       & \nuup(j_2) \ar[r]
       & 0
    }
    $$
\begin{deff}\label{def-normaldiff}
    The quotient vb-map $\nuup^{(2)}(F):= J^* f_{2*} / f_{1*}$ is called the \textbf{normal differential} of the 2-map of immersions $F$. When $F= (f_2,f_1)$ we also write $\nuup^{(2)}(f_2,f_1)$.
\end{deff}
Notice the occurence of the flip~(\S\ref{par.dbcat-flip}) in the above definition.
Let us start by examinating the behaviour of the normal differential with respect to vertical and horizontal compositions in $\Immer$.
\begin{prop}\label{prop-normalbundle-composition2}
Given a pair of horizontally composable 2-maps $F: i \Rightarrow j$ and $G : j \Rightarrow k$ in $\Immer$, we have
$
\nuup^{(2)}(G) \circ \nuup^{(2)}(F) = \nuup^{(2)}(G \circ F).
$
\end{prop}
\begin{proof}
  Let $I= (j,i)$ and $J=(k,j)$ be the vertical 2-immersions  associated to $F$ and $G$, respectively. Then, from $(J^*g_{2*})\circ(I^*f_{2*})= (J \bullet I)^*(g_{2*} \circ f_{2*})$, we obtain
 $$
 \frac{J^*g_{2*}}{g_{1*}} \circ \frac{I^*f_{2*}}{f_{1*}} 
 = \frac{(J \bullet I)^*(g_{2*} \circ f_{2*})}{g_{1*}\circ f_{1*}}
 =  \frac{(J \bullet I)^*(g_{2} \circ f_{2})_*}{(g_{1}\circ f_{1})_*}
 $$
 as expected, since $J \bullet I$ is the vertical 2-immersion associated to $G\circ F$.
\end{proof}
Now, given a composite of immersions $M \overset{i}{\looparrowright} N \overset{j}{\looparrowright} Q$, the normal bundle decomposes, albeit non-canonically, as
$\nuup(j \circ i) \cong \nuup(i)\oplus i^*\nuup(j).$
The same holds at the level of morphisms:
\begin{prop}\label{prop.vert-fun}
Let $H =(h_2,h_1): i_1 \Rightarrow i_2$ and $K=(k_2,k_1): j_1 \Rightarrow j_2$ be a pair of vertically composable 2-maps of immersions, i.e. $h_2 = k_1$,
 then there is a (non-canonical) isomorphism of vb-maps 
 $$
  \xymatrix@C=7pc@R=1pc{
    \nuup(j_1 \circ i_1) 
    \ar[r]^{\nuup^{(2)}(K \bullet H)}
    \ar@{}[d]|{\rotatebox[origin=c]{-90}{$\cong$}}
    & \nuup(j_2 \circ i_2)\ar@{}[d]|{\rotatebox[origin=c]{-90}{$\cong$}}
    \\
    i^*_{1}\nuup(j_1) \oplus \nuup(i_1) \ar[r]_-{I^*\nuup^{(2)}(K) \oplus \nuup^{(2)}(H)}& i_{2}^*\nuup(j_2) \oplus \nuup(i_2)
  }
  $$
\end{prop}
Since the proof of the latter proposition happens to be more technical than the former one, it will be delayed until the appendix~\ref{app.vert-fun}.
\begin{example}
    Let $\pi: E \rightarrow M$ be a smooth vector bundle with zero section $0_E: M \hookrightarrow E$. 
    Then the pair $(\pi,\id_M)$ defines a 2-map $\underline{\smash\pi}: 0_E \Rightarrow \id_M$ in $\Embed$. 
    In that case, the normal differential $\nuup^{(2)}(\underline{\smash \pi})$ identifies with $\pi$ through the vertical lift 
    $\nuup(0_E) \cong E$ and $\nuup(\id_M) = M$. 
    Analogously, let $\delta: \R \times E \rightarrow E$ be the $(\R,\cdot)$-action defined by $\delta_\lambda(m,\xi) = (m, \lambda \xi)$. 
    Then 
    \begin{align} 
    \begin{aligned}
    \xymatrix@C=1.5pc@R=1.5pc{
    M \times \R \ar[r]^-{\id_M} 
    \ar@{}[d]|{\rotatebox[origin=c]{90}{$\longleftarrow\joinrel\rhook$}}_-{0_E \times \id}
    \rRightarrow
    & M
    \ar@{}[d]|{\rotatebox[origin=c]{90}{$\longleftarrow\joinrel\rhook$}}^-{0_E}
    \\
    E \times \R
    \ar[r]_-{\delta}  
    & E
    }   
    \end{aligned}
    \end{align}
    defines a 2-map $\underline{\smash\delta}$ in $\Embed$. The normal differential $\nuup^{(2)}(\underline{\smash \delta})$  coincides with the action $\delta$ under the identifications $\nuup(0_E \times \id) \cong E \times \R$ and $\nuup(0_E)\cong E$ (through the vertical lifts). In particular this action is regular in the sense of~\cite{bursztyn-cabrera-delhoyo-2016}.
\end{example}
\subsubsection{Description of the normal bundle using homogeneity structure.}
Following~\cite{grabowski-rotkiewicz2009}, a vector bundle structure on a manifold $E$ is characterized by an action $\delta: \R \times E \rightarrow E$, $\delta(\lambda,e) = \delta_\lambda(e)$, of the monoid $(\R,\cdot)$ onto $E$ subject to the following regularity condition: for $e \in E$,
\begin{align}
     \ddtvar \delta(\lambda,e) =0 \quad \Rightarrow \quad e = \delta(0,e)
\end{align}
or equivalently, the vertical lift $\vlift_\delta: E \rightarrow TE$ is injective. The monoid action $\delta$ is called the \textbf{homogeneity structure} of the vector bundle $E$. A vector bundle structure with total space $E$ and homogeneity structure will summoned as a pair $(E, \delta)$, or a triple $(E,\delta,M)$ if one wants to precise the base manifold which, in that setting, is the submanifold $M = \delta_0(E)$ of $E$.
\begin{remark}
    We assume that $\delta_0: E \rightarrow M$ has global constant rank to guarantee that $\delta_0$ is a surmersion and $M$ is an embedded submanifold of $E$.
\end{remark}
In addition, within this framework, a vb-map $\varphi$ from $(E,\delta^E)$ to $(F, \delta^F)$ is nothing more than a $(\R,\cdot)$-equivariant smooth map $E \rightarrow F$, that is $\varphi$ intertwines the respective homogeneity action:
$
 \varphi \circ \delta^{E} = \delta^{F} \circ (\id_\R \times \varphi).
$
In particular, the notion of vb-immersion (resp. vb-embedding, etc...) coincides with $(\R,\cdot)$-equivariant immersion (resp. equivariant embedding, etc...).
%
%
\subsubsection{Pullback homogeneity structure.}\label{par.pb-homog-struct}
Let $(E,\delta, M)$ be a vector bundle with homogeneity action $\delta$, and $h: P \rightarrow M$ be a smooth map.
The pullback vector bundle $h^*E$ sits as a sub-bundle of the trivial bundle $P \times E \rightarrow P$. 
In particular, the restriction of the ambient homogeneity action $\id_P \times \delta: \R \times P \times E \rightarrow P \times E$ to $h^*E$ is well-defined and regular. Set the \textbf{pullback homogeneity structure} to be the $(\R,\cdot)$-action $h^*\delta: \R \times h^*E \rightarrow h^*E$ given  by the restriction
$h^*\delta := (\id_P \times \delta)\vert_{h^* E}.$
For an embedding $j: N \hookrightarrow M$, we sometimes use the notation $\delta\vert_N$ (instead of $j^*\delta$) and call it the restriction of the homogeneity action $\delta$ to $N$.
%
%
\subsubsection{Quotient homogeneity structure.}\label{par.quot-homog-struct}
\begin{prop}
   Let $(E, \delta^E, M)$ be a vector bundle and $F\hookrightarrow E$ be a wide vb-embedding. 
    The quotient vector bundle $E/F$ inherits a natural homogeneity structure, called the \textbf{quotient homogeneity structure}, defined by
    $
    \delta^{E/F}(\lambda,[e]) =  [\delta^E(\lambda,e)].
    $
\end{prop}
\begin{proof}
Without loss of generality we assume that $F \subseteq E$.
The well-definiteness follows from a direct computation.
To show the regularity, assume that 
$
\ddtvar \delta^{E/F}(\lambda,[e]) = 0,
$
then choose a smooth lift of the curve $\delta^{E/F}(-, [e]) : \R \rightarrow E/F$ into a curve $\delta^E(-,e) : \R \rightarrow E$. 
Our hypothesis implies that 
$
\ddtvar \delta^E(\lambda,e) \in VF\vert_M.
$
Thus, there exists $f \in F$ such that 
$
\ddtvar \delta^E(\lambda,e) = \ddtvar \delta^F(\lambda,f).
$
But since the vertical lift is an isomorphism, $e = f$ and $[e]$ lies in the zero section of $E/F$.
\end{proof}
\begin{remark}
 The quotient homogeneity is obtained as the quotient of homogeneity structure $\delta^E/\delta^F$ in the following sense: the following diagram commutes and the rows are exact,
 $$
 \xymatrix@R=1pc{
        \R \times 0 \ar[r] \ar[d]_-{\delta^{\triv}}
        &\R \times F  \xhookr[1.2pc] \ar[d]_-{\delta^F}
        &  \R \times E   \ar[r]  \ar[d]^-{\delta^E}
        & \R \times (E/F) \ar[d]^-{\delta^{E/F}} \ar[r]
        & \R \times 0 \ar[d]^-{\delta^{\triv}}
        \\ 
        0 \ar[r] 
        & F  \xhookr[2.2pc]
        &  E   \ar[r]
        & (E/F) \ar[r]
        & 0
    }
 $$
 where $\delta^{\triv}$ is the tautological homogeneity action onto the zero-dimensional vector bundle over $M$. Notice that the top row is a short exact sequence of bundle over $\R \times M$.
\end{remark}

\begin{example}[Normal bundle]
    Let $j: M \looparrowright N$ be an immersion of smooth manifolds. Since the normal bundle $\nuup(j)=j^*TN/TM $ is defined as the quotient of a pullback vector bundle, its homogeneity structure is described by 
    $
    \delta^{\nuup(j)}(\lambda, [m, \eta]) = [m, \delta^{TN}(\lambda,\eta)]$ for $m \in M$ and $\eta \in T_{j(m)}N$.
\end{example}
%
%
%
\section{Double vector bundles}\label{sec.dvb}
Following~\cite{grabowski-rotkiewicz2009}, a \textbf{double vector bundle} $(D,\delta^h, \delta^v)$ is a smooth manifold $D$ equipped with two commuting smooth regular $(\R,\cdot)$-actions $\delta^h$ and $\delta^v$, respectively called the horizontal action and the vertical action. The vector bundle projections $\pi^h$ and $\pi_v$, defined by $\pi^h(d) = \delta^h(0,d)$ and $\pi^v(d) = \delta^v(0,d)$, get organized into a commuting square (which we encode by the corresponding  ``matrix'' as shown on the right-hand side below):
$$
\left. \vcenter{
\xymatrix@C=1.2pc@R=1.2pc{
    D \ar[r]^{\pi^h} \ar[d]_{\pi^v}
    & A \ar[d]^{\pi^v \vert_A} 
    \\
    B \ar[r]_{\pi^h\vert_B} & M} 
} \right.
\quad =: \quad 
\left[
    \begin{matrix}
        D & A \\ B & M
    \end{matrix}
\right]
$$
Alternatively, following~\cite{mackenzie2005}, we write $(D,A,B,M)$ in order to summon a double vector bundle by precising its side bundle $A$ and $B$. The corresponding homogeneity structure consists in the scalar multiplications of the vector bundles $D \rightarrow A$ and $D\rightarrow B$.
If $A=B$, then the double vector bundle is said to be symmetric.
The (total) \textbf{flip} of $(D,\delta^h, \delta^v)$ is the double vector bundle $\flip(D)= (D,\delta^v, \delta^h)$, in which the $(\R,\cdot)$-actions have been reversed.
As double vector bundles, we distinguish $D$ from $\flip(D)$ and will treat $\flip: D \rightarrow \flip(D)$ as a special kind of morphism of double vector bundles. 
In contrast, a (non-flipping) \textbf{dvb-map} $\varphi$ is a smooth map $\varphi$ between the total space intertwining the respective horizontal homogeneity actions, as well as the vertical ones. That is, $\varphi: (D_1,A_1,B_1, M_1) \rightarrow (D_2,A_2,B_2, M_2)$ is a dvb-map if $\varphi: (D_1,A_1) \rightarrow (D_2,A_2)$ and $\varphi: (D_1,B_1) \rightarrow (D_2,B_2)$ are vb-maps such that $\varphi\vert_{A_1}: (A_1,M_1) \rightarrow (A_2,M_2)$ and $\varphi\vert_{B_1}: (B_1,M_1) \rightarrow (B_2,M_2)$ also are vb-maps.
In practice, a dvb-map will be viewed as a tuple of vb-maps $(\varphi,\alpha, \beta, f)$, with $\alpha = \varphi\vert_{A_1}$, $\beta = \varphi\vert_{B_1}$, and $f = \varphi\vert_{M_1}$, organized in a ``matrix'' as follows:
\begin{align*}
 \begin{split}
    \left[
        \vcenter{
        \xymatrix@R=0pc@C=0pc{
            D_1  & A_1 \\ B_1 & M_1
        }}    
        \right]
    \xrightarrow{ \quad 
        \left[\begin{smallmatrix}
         {\varphi} & \alpha 
        \\ {\beta} & f 
        \end{smallmatrix}\right]
        \quad
        }
        \left[
        \vcenter{
        \xymatrix@R=0pc@C=0pc{
            D_2
            & A_2 
            \\ B_2  
            & M_2
        } }
        \right]
 \end{split}
\end{align*}
The vb-maps $\alpha$ and $\beta$ will respectively be called horizontal and vertical \textbf{side-maps} of the dvb-map $\varphi$.
The category of double vector bundles equipped with non-flipping dvb-maps will be denoted by $\DVB$, and the associated double category of commutative squares by $\DVB^\square$.
In fact, the most general kind of dvb-map that we will encounter in the sequel are, roughly speaking, composites of non-flipping dvb-maps and flip isomorphisms, but we reserve the appellation ``dvb-maps'' for the non-flipping ones.
A \textbf{dvb-immersion} (resp. dvb-embedding) is a dvb-map as before such that $\varphi, \alpha, \beta,f$ are immersions (resp. embedding). Similarly, a \textbf{dvb-submersion} (resp. dvb-surmersion) is a dvb-map $\varphi$ such that $\varphi, \alpha, \beta,f$ are submersion (resp. surmersion). In particular, if $\varphi$ is a dvb-immersion (resp. dvb-submersion), then each of the vb-maps $\varphi, \alpha, \beta$ are fiberwise injective (resp. surjective).
By convention, we assume that submersions and immersions have global constant rank. Along the same lines, a constant-rank dvb-map is a dvb-map with constant-rank components $\varphi, \alpha,\beta, f$.
%
\begin{example}\label{ex.dvb-TE}
 Let $E \rightarrow M$ be a vector bundle with homogeneity structure $\delta$. The tangent bundle $TE \rightarrow E$ has a canonical homogeneity action $\delta^{T\!E}_E$ (the horizontal one) which turns out to be compatible with the tangent action $\delta_*: \R \times TE \rightarrow TE$. 
 It results that $(TE, \delta^{T\!E}_E, \delta_*)$ is a double vector bundle with horizontal side bundle $E$ and vertical side bundle $TM$. In particular, the double tangent bundle $TTM$ of a smooth manifold $M$ admits a dvb-structure.
\end{example}
\begin{example}\label{ex.f**}
 Let $f: M \rightarrow N$ be a smooth map, then the double differential $f_{**}: TTM \rightarrow TTN$ is a dvb-map~\cite{fisher-laquer1999}. More generally, if $\varphi: E \rightarrow F$ is a vb-map, then its differential defines dvb-map $\varphi_*: TE \rightarrow TF$ where $TE$ and $TF$ are equipped with the dvb-structure of the previous example.
\end{example}

\begin{example}[Direct product]
    Let $(D, \delta^h, \delta^v)$ and $(D', \kappa^h, \kappa^v)$ be two double vector bundles, then $(D \times D', \delta^h \times \kappa^h, \delta^v \times \kappa^v)$ is also a double vector bundle when equipped with the diagonal actions, namely $\delta^h \times \kappa^h: \R \times D \times D' \rightarrow D \times D'$ is defined as $(\delta^h \times \kappa^h)(\lambda, d,d') = (\delta^h(\lambda,d), \kappa^h(\lambda,d'))$, and analogously for the vertical action. 
    The (direct) product double vector bundle $D \times D'$ turns out to have various associated flip in addition to the total flip, referred as \textbf{partial flip}:  $D \times \flip(D')$ and $\flip(D) \times D'$ with respective homogeneity structure $(\delta^h \times \kappa^v, \delta^v \times \kappa^h)$ and $(\delta^v \times \kappa^h, \delta^h \times \kappa^v)$. 
    More generally, a $n$-fold product double vector bundle $D_1 \times \dots \times D_n$ has $2^{n-1}$ distinct, say strictly, partial flips (excluding the trivial identity flip and the total flip).
\end{example}
In order to emphasize the distinction between the non-flipping dvb-map and the flipping-ones, we introduce the following - empirical - notion:
\begin{deff}\label{def.flip-map}
 Let $D \subseteq D_1 \times \dots \times D_n$ and $D' \subseteq D'_1 \times \dots \times D'_n$ be two double vector sub-bundles.
 A \textbf{flip isomorphism} $D \flipar D'$ (resp. flip map) is the composite of the restriction of a (trivial, partial or total) flip from the ambient double vector bundle, with a dvb-isomorphism (resp. dvb-map).
 The double vector bundle equipped with composite of flip maps and dvb-maps forms a category denoted $\DVB_f$.
\end{deff}
%
%
%
\subsubsection{Fiber products}
A cospan in $\DVB$, or shortly a dvb-cospan,
\begin{align}\label{dvb.cospan}
    \left[
    \vcenter{
    \xymatrix@C=0pc@R=0pc{
    D_1  & A_1 \\ B_1  & M_1
    }}
    \right]
    \xrightarrow{
        \left[
        \begin{smallmatrix}
            \varphi_1 & \alpha_1 \\ \beta_1 & f_1
        \end{smallmatrix}
        \right]
        } 
    \left[
    \vcenter{
    \xymatrix@C=0pc@R=0pc{
    D & A  \\ B & M
    }}
    \right]
    \xleftarrow{
        \left[
        \begin{smallmatrix}
            \varphi_2 & \alpha_2 \\ \beta_2 & f_2
        \end{smallmatrix}
        \right]
        } 
    \left[
    \vcenter{
    \xymatrix@C=0pc@R=0pc{
    D_2 & A_2  \\ B_2  & M_2
    }}
    \right]
\end{align}
is called \textbf{good} if: 
\begin{itemize}
\setlength\itemsep{0pc}
 \item[(i)] $(M_1 \xrightarrow{f_1} M \xleftarrow{f_2} M_2)$ is a good cospan in $\Smooth$;
 \item[(ii)] $(A_1 \xrightarrow{\alpha_1} A \xleftarrow{\alpha_2} A_2)$ is a good cospan in $\VB$ over $(M_1 \xrightarrow{f_1} M \xleftarrow{f_2} M_2)$;
 \item[(iii)] $(B_1 \xrightarrow{\beta_1} B \xleftarrow{\beta_2} B_2)$ is a good cospan in $\VB$ over $(M_1 \xrightarrow{f_1} M \xleftarrow{f_2} M_2)$;
 \item[(iv)] $(D_1 \xrightarrow{\varphi_1} D \xleftarrow{\varphi_2} D_2)$ is a good cospan in $\VB$ over $(A_1 \xrightarrow{\alpha_1} A \xleftarrow{\alpha_2} A_2)$;
  \item[(v)]$(D_1 \xrightarrow{\varphi_1} D \xleftarrow{\varphi_2} D_2)$ is a good cospan of $\VB$ over $(B_1 \xrightarrow{\beta_1} B \xleftarrow{\beta_2} B_2)$.
\end{itemize}
\begin{example}
 If $\varphi_1$ is a dvb-submersion then the cospan~(\ref{dvb.cospan}) is good.
\end{example}

\begin{prop}\label{prop.dvb-fp}
 Let $(D_1 \rightarrow D \leftarrow D_2)$ be a good cospan of double vector bundles as above. 
 Then, the fiber product $D_1 \times_D D_2$ is a well-defined double vector bundle:
 $$
    \vcenter{
    \xymatrix@C=1pc@R=1pc{
    D_1 \times_D D_2 \ar[r] \ar[d] & A_1 \times_M A_2 \ar[d] \\ B_1 \times_M B_2 \ar[r] & M_1 \times_M M_2
    }}
 $$
\end{prop}
\begin{proof}
 Let $(\delta_1^h, \delta^v_1)$, $(\delta_2^h, \delta^v_2)$ and $(\delta^h, \delta^v)$ be the homogeneity structure of $D_1, D_2$ and $D$ respectively. We claim that taking the fiber product of the homogeneity structure returns a homogeneity structure $(\delta_1^h \times_{\delta^h} \delta^h_2, \delta_1^v \times_{\delta^v} \delta^v_2)$. Indeed, the action $\delta_1^h \times_{\delta^h} \delta^h_2: \R \times (D_1 \times_D D_2) \rightarrow D_1 \times_D D_2$ is regular: 
 assume that $(d_1,d_2) \in D_1 \times_D D_2$ with
 $
 \ddtvar (\delta_1^h \times_{\delta^h} \delta^h_2)(\lambda,d_1,d_2)  \: = \: 0.
 $
 Since $\delta_1^h \times_{\delta^h} \delta^h_2$ is the restriction of the $(\R,\cdot)$-action $\delta_1^h \times \delta^h_2$  to the submanifold $D_1 \times_D D_2$, this implies
 $\ddtvar (\delta_1^h(\lambda, d_1)) =0$ and 
 $\ddtvar (\delta_1^h(\lambda, d_1)) =0.$
 Thus by regularity, $d_1 = \delta_1^h(0, d_1)$ and $d_2 = \delta_2^h(0, d_2)$, that is $(d_1,d_2) = (\delta_1^h \times_{\delta^h}\delta_2^h)(0, d_1, d_2)$. By the same argument, the other involved $(\R,\cdot)$-actions are again regular. 
\end{proof}
As in $\VB$, the fiber product in $\DVB$ is natural in the following sense: a morphism of good dvb-cospan $(Q_1 \!\rightarrow \! Q \!\leftarrow\! Q_2) \rightarrow (D_1 \!\rightarrow \! D \!\leftarrow \!D_2)$ induces a dvb-map between the respective fiber products 
$
  Q_1 \times _Q Q_2 \rightarrow D_1 \times_D D_2.
$
%
%
%
\subsection{Side-pullback}
\subsubsection{Horizontal and vertical lift of a vector bundle}\label{par.dvb-lift}
Let $ \pi :E \rightarrow M$ be a vector bundle with homogeneity structure $\delta: \R \times E \rightarrow E$. Then, we define the \textbf{vertical dvb-lift} and the \textbf{horizontal dvb-lift} of $E$ to be the following double vector bundles:
\begin{align}
    \begin{split}
        E^v\: = \: 
        \left[
        \vcenter{
        \xymatrix@R=0pc@C=0pc{
            E  & M  \\ E  & M
        }}    
        \right],
        \qquad 
         E^h \: = \:
        \left[
        \vcenter{
        \xymatrix@R=0pc@C=0pc{
            E  & E  \\ M  & M
        } }
        \right].
    \end{split}
    \end{align}
Their respective homogeneity structures are $(\delta, \id)$ and $(\id, \delta)$. Clearly, $\flip(E^v) = E^h$.
If $\varphi:E \rightarrow F$ is a vb-map, then it induces dvb-maps $\varphi^h: E^h \rightarrow F^h$ and $\varphi^v: E^v \rightarrow F^v$ in the obvious way.
Moreover, given a double vector bundle $(D,A,B,M)$, there are two canonical dvb-surmersion $\underline{\pi}^h : D \rightarrow A^h$ and $\underline{\pi}^v: D \rightarrow B^v$ respectively given by
$$
    \left[\vcenter{
    \xymatrix@R=0pc@C=0pc{
        D  & A \\ B  & M
    }}\right] 
    \xrightarrow{ \:
    \left[
    \begin{smallmatrix}
     \pi^h & \id \\ \pi & \id
    \end{smallmatrix}
    \right]
    \:
    }
     \left[\vcenter{
    \xymatrix@R=0pc@C=0pc{
        A  & A \\ M & M
    }}\right]
    \quad 
    \text{ and }
    \quad
    \left[\vcenter{
    \xymatrix@R=0pc@C=0pc{
        D  & A \\ B & M
    }}\right] 
    \xrightarrow{ \:
    \left[
    \begin{smallmatrix}
     \pi^v & \pi \\ \id & \id
    \end{smallmatrix}
    \right]
    \:
    }
     \left[\vcenter{
    \xymatrix@R=0pc@C=0pc{
        B & M  \\ B & M
    }}\right].
$$
\begin{remark}\label{rem.dvb-lift}
 For any double vector bundle $D$ such that $A = B$ (in particular, for $D= TTM$), we have $ \flip \circ \underline{\pi}^h = \underline{\pi}^v$. Moreover, if $\varphi: E \rightarrow F$ is a vb-map, then $\flip \circ \varphi^h \circ\flip = \varphi^v$. 
\end{remark}

%
%
\subsubsection{}
Let $(D,A,B,N)$ be a double vector bundle with homogeneity structure $(\delta^h,\delta^v)$. 
Given a vector bundle $E \rightarrow M$ with homogeneity structure $\delta$ and a vb-map $(\varphi,f): (E,M) \rightarrow (A, N)$, the fiber product $\varphi^*D = E \times_{A} D $ in $\Smooth$ sits as a submanifold of $E \times D$.
Now, notice that the cospan $(\varphi^h,\underline{\pi}^h): (E^h \rightarrow A^h \leftarrow D)$ is good, hence $\varphi^*D$ inherits a dvb-structure as a double vector sub-bundle of the direct product $E^h \times D$. 
The \textbf{horizontal side-pullback} of $D$ along $\varphi$ is the double vector bundle $\varphi^{*,h}D = (\varphi^*D, E, f^*B, M)$ where the dvb-structure is induced by the embedding
\begin{align}
\begin{split}
\left[
\vcenter{
    \xymatrix@C=0pc@R=0pc{
    \varphi^*D   & E 
    \\
    f^*B & M
    } }
\right]
\quad
\lhook\joinrel\xrightarrow{\quad}
\quad
    \left[
        \vcenter{
        \xymatrix@R=0pc@C=0pc{
            E & E \\ M & M
        } }
    \right]
    \: \times \: 
    \left[
        \vcenter{
        \xymatrix@R=0pc@C=0pc{
            D & A \\ B & N
        } }
    \right]
        \end{split}
\end{align}
given by 
$
\left[
\begin{smallmatrix}
           (e,d) & e \\ (m,b) & m
\end{smallmatrix} 
\right]
\mapsto
\left[
\begin{smallmatrix}
           e & e \\ m & m
\end{smallmatrix}
\right]
\times
\left[
\begin{smallmatrix}
           d & \varphi(e) \\ b & f(m)
\end{smallmatrix}
\right].
$
The homogeneity structure on $\varphi^{*}D$, denoted $\varphi^{*,h}(\delta^h, \delta^v)$, is obtained as the pullback of the product homogeneity structure $(\id \times \delta^h, \delta \times \delta^v)$.
%
%
\subsubsection{}
Analogously, let $(D,A,B,N)$ be a double vector bundle with homogeneity structure $(\delta^h,\delta^v)$, and let $(\psi,g): (E,M) \rightarrow (B, N)$ be a vb-map. The submanifold $\psi^*D$ of $E \times D$ inherits a dvb-structure as a double vector sub-bundle of the direct product $E^v \times D$. The \textbf{vertical side-pullback} of $D$ along $\psi$ is the double vector bundle $\psi^{*,v}D = (\psi^*D, g^*A, E,M)$ with dvb-structure induced by the embedding 
\begin{align}
\begin{split}
\left[
\vcenter{
    \xymatrix@C=0pc@R=0pc{
    \psi^*D  & g^*A
    \\
    E  & M
    } }
\right]
\quad
\lhook\joinrel\xrightarrow{\quad}
\quad
    \left[
        \vcenter{
        \xymatrix@R=0pc@C=0pc{
            E & M \\ E & M
        } }
    \right]
    \: \times \: 
    \left[
        \vcenter{
        \xymatrix@R=0pc@C=0pc{
            D & A \\ B  & N
        } }
    \right]
        \end{split}
\end{align}
given by $
\left[
\begin{smallmatrix}
           (e,d) & (m,a) \\ e & m
\end{smallmatrix} 
\right]
\mapsto
\left[
\begin{smallmatrix}
           e & m \\ e & m
\end{smallmatrix}
\right]
\times
\left[
\begin{smallmatrix}
           d & a \\ \psi(e) & g(m)
\end{smallmatrix}
\right].
$
The homogeneity structure $\psi^{*,v}(\delta^h, \delta^v)$ on $\psi^{*,v}D$ is the pullback of the product homogeneity structure $(\delta \times \delta^h, \id \times \delta^v)$. 
\begin{remark}
 Strictly speaking, the fiber product double vector bundle associated to the cospan $(\varphi^h,\underline{\pi}^h): (E^h \rightarrow A^h \leftarrow D)$ is a double vector bundle $(E \times_{\varphi,\pi^h} D, M \times_{\id, \pi} E, M \times_{f,\pi} B, M \times_{f,\id} N)$. It is naturally isomorphic to the horizontal side-pullback, as defined above, through the natural isomorphisms $M \times_{\id, \pi} E \cong E$ and $M \times_{f,\id} N \cong M$. In particular, the horizontal zero section $E \hookrightarrow \varphi^{*}D$ is given by the graph embedding of $\varphi$, whereas the vertical zero section $f^*B \hookrightarrow \varphi^{*}D$ is given by the fiber product vb-map $0_M^E \times_{0_N^A} 0_B^D$ of the corresponding zero-sections.
\end{remark}

\begin{example}\label{ex.sidepb-dbtg}
 Consider the double tangent bundle $(TTN,\delta^h, \delta^v)$ and let $f: M \rightarrow N$ be a smooth map. Then the side-pullbacks $(f_*)^{*,h} TTN$ and $(f_*)^{*,v} TTN$ are respectively identified to some double vector sub-bundles of the direct products $TM^h \times TTN$ and $TM^v \times TTN$. 
 In particular, the vertical side-pullback $(f_*)^{*,v}TTN$ is naturally isomorphic to $T(f^* TN)$ (with the double vector bundle structure of example~\ref{ex.dvb-TE}). Indeed, since $\pi^v = \pi_*$ and the cospan $(f,\pi):(M \rightarrow N \leftarrow TN)$ is good, there is a natural vb-isomorphism (as vector bundle over $f^*TM$)
\begin{align}\label{eq.vert-id}
 T(f^*TN)  
 \cong (f_*)^{*,v}TTN
\end{align}
which is compatible with the vertical vb-structure over $TM$.
As a conventional notation, we will make the following abuse (when the context allows it),
\begin{align}\label{eq.convention}
\begin{array}{rcl}
    (f_*)^* TTN &:=& (f_*)^{*,h}TTN.
\end{array}
\end{align}
That being said, the total flip $TM^h \flipar TM^v$ induces a partial flip $TM^h \times TTN \flipar TM^v \times TTN$, whose restriction together with the natural isomorphism~(\ref{eq.vert-id})  yields the flip isomorphism, thanks to remark~\ref{rem.dvb-lift}:
\begin{align}\label{Phlip}
     \Phi= \Phi_f: (f_*)^*TTN \flipar T(f^*TN).
\end{align}
Its inverse will be denoted with an upper index, $\Phi^f = \Phi_f^{-1}$.
\end{example}
%
\subsubsection{Side-pullbacks of dvb-maps by 2-vb-maps.}
Let $(\varphi,\alpha,\beta,f)$ be a dvb-map $(D_1, A_1, B_1, M_1) \rightarrow (D_2, A_2, B_2, M_2)$ . Given a 2-vb-map
\begin{align*}
 \begin{split}
 \left.
    \vcenter{
    \xymatrix@R=1.5pc@C=1.5pc{
        E_1 \ar[r]^-{\psi_1} \ar[d]_-\gamma 
        \ar@{}[rd]|-{\overset{{\Psi}}{\rotatebox[origin=c]{0}{$\Rightarrow$}}}
        & A_1 \ar[d]^-\alpha
        \\
        E_2 \ar[r]_-{\psi_2} & A_2
    }}
    \right.
    \qquad
     \text{with base 2-map}
    \qquad
    \left.
    \vcenter{
    \xymatrix@R=1.5pc@C=1.5pc{
        N_1 \ar[r]^-{h_1} \ar[d]_-g 
        \ar@{}[rd]|-{\overset{{H}}{\rotatebox[origin=c]{0}{$\Rightarrow$}}}
        & M_1 \ar[d]^-f
        \\
        N_2 \ar[r]_-{h_2} & M_2
    }}
    \right.
 \end{split}
\end{align*}
 we define the \textbf{horizontal side-pullback} $\Psi^{*,h} \varphi: \psi_1^{*,h}D_1 \rightarrow \psi_2^{*,h}D_2$ of $\varphi$ along $\Psi$ to be the dvb-map
\begin{align*}
 \begin{split}
    \left[
        \vcenter{
        \xymatrix@R=0pc@C=0pc{
            \psi_1^*D_1 & E_1 \\ h_1^*B_1 & N_1
        }}    
        \right]
        \xrightarrow{ \quad 
        \left[\begin{smallmatrix}
         \Psi^*\!\varphi \,& \gamma 
        \\ H^*\!\beta \,& g 
        \end{smallmatrix}\right]
        \quad
        }
        \left[
        \vcenter{
        \xymatrix@R=0pc@C=0pc{
            \psi_2^*D_2
            & E_2  
            \\ h_2^*B_2 
            & N_2
        } }
        \right]
 \end{split}
\end{align*}
where $\Psi^*\varphi$ (here $\varphi$ covers $\alpha$) and $H^*\beta$ are pullbacks of vb-maps along 2-maps (cf.~\S\ref{par.pb-vbmap}). 
In particular, the horizontal side-pullback $\Psi^{*,h}\varphi$ identifies in a natural way to the restriction of the product dvb-map 
\begin{align*}
 \begin{split}
    \left[
        \vcenter{
        \xymatrix@R=0pc@C=0pc{
            E_1 & E_1  
            \\ N_1 & N_1
        }}    
    \right]
    \times
    \left[
        \vcenter{
        \xymatrix@R=0pc@C=0pc{
            D_1 & A_1 
            \\ B_1 & M_1
        }}    
    \right]
    \xrightarrow{ \quad 
    \left[\begin{smallmatrix}
         \gamma & \gamma 
        \\ g & g 
        \end{smallmatrix}
    \right]
    \times 
    \left[\begin{smallmatrix}
         \varphi \,& \alpha 
        \\ \beta \,& f 
        \end{smallmatrix}
    \right]
        \quad
        }
    \left[
        \vcenter{
        \xymatrix@R=0pc@C=0pc{
            E_2  & E_2 
            \\ N_2  & N_2
        }}    
    \right]
    \times
    \left[
        \vcenter{
        \xymatrix@R=0pc@C=0pc{
            D_2
            & A_2 
            \\ B_2  
            & M_2
        } }
    \right].
 \end{split}
\end{align*}
\begin{remark}\label{rem.natural}
 Reformulating the last sentence, $\Psi^{*,h}\varphi$ is, up to natural isomorphism of the source and the range,    given by the fiber product dvb-map $\gamma^h \times_{ \alpha^h} \varphi$ associated to the cospan $(\Psi^h, \underline{\smash \pi}^h): \gamma^h \Rightarrow \alpha^h \Leftarrow \varphi$ where $\Psi^h$ is the obvious lift of $\Psi$ as a 2-dvb-map. 
We write it as $\Psi^{*,h}\varphi \cong \gamma^h \times_{ \alpha^h} \varphi$ in order to distinguish it from a strict equality.
\end{remark}
The vertical side-pullback of dvb-maps is defined along the same lines. 
\begin{example}
 Let $H= (h_2,h_1): (N_1 \xrightarrow{g} N_2) \Rightarrow (M_1 \xrightarrow{f} M_2)$ be a 2-map in $\Smooth$. 
 Applying the tangent functor returns $H_* = (h_{2*}, h_{1*})$ as a 2-vb-map $g_* \Rightarrow f_*$. 
 On one side, the horizontal side-pullback $(H_*)^{*,h}f_{**}$ is the dvb-map 
 \begin{align*}
 \begin{split}
    \left[
        \vcenter{
        \xymatrix@R=0pc@C=0pc{
            (h_{1*})^*TTM_1 & TN_1 
            \\ h_1^*TM_1  & N_1
        }}    
        \right]
        \xrightarrow{ \quad 
        \left[\begin{smallmatrix}
            (H_*)^*\!f_{**} \,& g_* 
            \\ H^*\!f_* \,& g 
        \end{smallmatrix}\right]
        \quad
        }
        \left[
        \vcenter{
        \xymatrix@R=0pc@C=0pc{
            (h_{2*})^*TTM_2
            & TN_2 
            \\ h_2^*TM_2 
            & N_2
        } }
        \right]
 \end{split}
\end{align*}
On the other side, the vertical side-pullback $(H_*)^{*,v}f_{**}$ naturally identifies with the dvb-map 
 \begin{align*}
 \begin{split}
    \left[
        \vcenter{
        \xymatrix@R=0pc@C=0pc{
            T(h_1^*TM_1)  & h_1^*TM_1 
            \\  TN_1 & N_1
        }}    
        \right]
        \xrightarrow{ \quad 
        \left[\begin{smallmatrix}
            (H^*f_*)_*\,&  H^*\!f_*
            \\ g_* \,& g 
        \end{smallmatrix}\right]
        \quad
        }
        \left[
        \vcenter{
        \xymatrix@R=0pc@C=0pc{
            T(h_2^*TM_2)
            & h_2^*TM_2 
            \\ TN_2 
            & N_2
        } }
        \right]
 \end{split}
\end{align*}
through $(h_1)^{*,v}TTM_1 \cong T(h_1^*TM_1)$ and $(h_2)^{*,v}TTM_2 \cong T(h_2^*TM_2)$.
\end{example}
\begin{remark}
 In the spirit of the notation~(\ref{eq.convention}), by writting $(H_*)^*f_{**}$ we always assume that the side-pullback is in the horizontal direction, unless explicitly specified by a superscript ${\_}^{*,v}$. On the other side, $(H^*f_*)_*$ is understood as the vertical side-pullback under the natural identification $H^{*,v}f_{**} \cong (H^*f_*)_*$.
\end{remark}

\subsubsection{Side-sharpening of dvb-maps.}
Consider a dvb-map $(\varphi,\alpha,\beta,f): (D_1, A_1, B_1, M_1) \rightarrow (D_2, A_2, B_2, M_2)$. Then, one defines the \textbf{horizontal sharpening} of $\varphi$ as the dvb-map $\varphi_{\dag}^h: D_1 \rightarrow \alpha^{*,h}D_2$ given by
\begin{align*}
 \begin{split}
    \left[
        \vcenter{
        \xymatrix@R=0pc@C=0pc{
            D_1 & A_1 \\ B_1 & M_1
        }}    
        \right]
        \xrightarrow{ \quad 
        \left[\begin{smallmatrix}
         \varphi_\dag^h & \id 
        \\ \beta_\dag & \id 
        \end{smallmatrix}\right]
        \quad
        }
        \left[
        \vcenter{
        \xymatrix@R=0pc@C=0pc{
            \alpha^*D_2 & A_1  
            \\ f^*B_2 & M_1
        } }
        \right]
 \end{split}
\end{align*}
  where $\varphi_\dag^h = \dtimes{\varphi}{\pi^h}$ and $\beta_\dag= \dtimes{\beta}{\pi}$ are the usual sharpening as vb-maps over $\alpha$ and $f$, respectively. In particular, if $\varphi$ has constant rank, then $\alpha^{*,h}\varphi$ has constant rank as well.
One easily check that ${\varphi}_\dag^h$ is compatible with the homogeneity structure using the fact that it is essentially the restriction of a product dvb-map.
%
%
%
\begin{remark}
 In the same manner, one defines the vertical sharpening $\varphi_\dag^v: D_1 \rightarrow \beta^{*,v}D_2$ by applying the vertical side-pullback instead of the horizontal one.
\end{remark}
\begin{remark}\label{rem.sharp-as-pb}
 The dvb-map $\varphi^h_\dag$ is essentially an instance of a horizontal side-pullback of a dvb-map by a 2-vb-map: precisely, let $\underline{\smash\alpha}:  \id_{A_1} \Rightarrow \alpha$ be the range lift of $\alpha$, given by the pair $\underline{\smash\alpha}=(\alpha, \id_{A_1})$. Then $\varphi^h_\dag \cong \underline{\smash\alpha}^{*,h} \varphi$. According to remark~\ref{rem.natural}, $\varphi^h_\dag$ also identifies with the fiber product of the good cospan 
 $(\underline{\smash\alpha},\underline{\smash\pi}^h): 
 ( \id_{A_1} \Rightarrow \alpha \Leftarrow \varphi).
 $
 Similarly, for the vertical sharpening $\varphi^v_\dag \cong \underline{\smash\beta}^{*,v}\varphi$.
\end{remark}
%
%
%
\begin{example}\label{ex.sharp-dvb-diff}
 Let $f: M \rightarrow N$ be a smooth map and $f_{**} : TTM \rightarrow TTN$ be its double differential. Recall that the range lift $\underline{f_*}$ is the 2-vb-map $\id_{T M} \Rightarrow f_*$ given by the pair $(f_*,\id_{TM})$.
 We define the \textbf{horizontally-sharp double differential} $f_{*\sharp}$ to be the dvb-map:
 \begin{align*}
 (\underline{f_*})^{*,h}f_{**} : \quad 
 \begin{split}
    \left[
        \vcenter{
        \xymatrix@R=0pc@C=0pc{
            TTM  & TM  \\ TM & M
        }}    
        \right]
    \xrightarrow{ \quad 
        \left[\begin{smallmatrix}
         f_{*\sharp} & \id 
        \\ f_\sharp & \id 
        \end{smallmatrix}\right]
        \quad
        }
        \left[
        \vcenter{
        \xymatrix@R=0pc@C=0pc{
            (f_*)^*TTN 
            & TM  
            \\ f^*TN 
            & M
        } }
        \right]
 \end{split}
\end{align*}
Regarding the vertical counterpart, the \textbf{vertically-sharp double differential} $(f_*)_{\sharp,v}= (\underline{f_*})^{*,v}f_{**}$ is naturally identified with the dvb-map $f_{\sharp *}$ defined as
\begin{align*}
(\underline{ f}^*f_*)_*: \quad 
 \begin{split}
    \left[
        \vcenter{
        \xymatrix@R=0pc@C=0pc{
            TTM  & TM \\ TM  & M
        }}    
        \right]
    \xrightarrow{ \quad 
        \left[\begin{smallmatrix}
         f_{\sharp *} & f_\sharp 
        \\ \id & \id 
        \end{smallmatrix}\right]
        \quad
        }
        \left[
        \vcenter{
        \xymatrix@R=0pc@C=0pc{
            T(f^*TN)
            & f^*TN
            \\  TM 
            & M
        } }
        \right]
 \end{split}
\end{align*}
\end{example}

\begin{prop}\label{prop.Phlip-map}
 $\Phi \circ f_{* \sharp} = f_{\sharp *}$ using the notations above and from~(\ref{Phlip}). 
\end{prop}
\begin{proof}
Let us use the generic notation $\underline{\smash\pi}^h,\underline{\smash \pi}^v$ for the dvb-lifts of the corresponding horizontal and vertical projections, respectively.
From remark~\ref{rem.sharp-as-pb}, the horizontal sharpening
 $
 f_{* \sharp} = (\underline{f_*})^{*,h} f_{**} \cong \underline{\pi}^h\,{}_{\underline{f_*}}\!\!\times_{\underline{\pi}^h} f_{**} 
 $
 is essentially the corestriction of the dvb-map $\dtimes{\underline{\pi}^h}{f_{**}}: TTM\rightarrow TM^h \times TTN$ over the fiber product $TM \,{}_{f_*}\!\!\times_{\pi^h} TTN$. Now notice that, since $\Phi$ is induced by $\flip \times \id : TM^h \times TTN \rightarrow TM^v \times TTN$, the identity of dvb-maps $(\flip \times \id) \circ \dtimes{\underline{\smash\pi}^h}{f_{**}} = \dtimes{\underline{\smash\pi}^v}{f_{**}}$ corestricts to 
 $
 \Phi \circ (\underline{\pi}^h\,{}_{\underline{f_*}}\!\!\times_{\underline{\pi}^h} f_{**}) = \underline{\pi}^v\,{}_{\underline{f_*}}\!\!\times_{\underline{\pi}^v} f_{**}
 $
 over $TM \,{}_{f_*}\!\!\times_{\pi^v} TTN$.
\end{proof}
%
%
%
\subsubsection{Flip isomomorphism.}
%
\begin{lem}[Flip lemma]\label{lem.flip}
 Let $(D,A,B,M)$ be a double vector bundle and $(\varphi,f):(E,N) \rightarrow (A,M)$ be a vb-map. Then, there is a canonical dvb-isomorphism
 $
 \flip( \varphi^{*,h}D) \cong \varphi^{*,v}\flip(D).
 $
\end{lem}
\begin{proof}
 Let $\delta$ and $(\delta^h,\delta^v)$ be the homogeneity structure on $E$ and $D$ respectively.
 As manifolds, $\varphi^{*,h}D = E \times_{\pi, \pi^h} D$ is sent to $\varphi^{*,v}\flip(D)= E \times_{\pi, \pi^v} D$ under the total flip $E^h \times D \cong \flip(E^h \times D)$.
 Now, $ \flip(E^h \times D) =  \flip(E^h) \times \flip(D) =  E^v \times \flip(D)$ has homogeneity structure $(\delta \times \delta^v, \id \times \delta^h)$, thus matching the homogeneity structure on the vertical pullback of $\flip(D)$. 
\end{proof}
\begin{deff}\label{def.flipmap}
 Let $\varphi: D_1 \rightarrow D_2$ be a dvb-map of the form 
        $\left[
        \begin{smallmatrix}
         \varphi& \alpha 
        \\ \beta & f
        \end{smallmatrix}
        \right]
        $.
        Then the \textbf{flip of $\varphi$} is the dvb-map
        $
        \flip(\varphi): \flip(D_1) \rightarrow \flip(D_2)
        $ 
        given by 
        $$
        \flip(\varphi) = 
        \flip_2 \circ \left[
        \begin{matrix}
        \varphi  & \beta
        \\ \alpha & f 
        \end{matrix}
        \right] \circ \flip_1.
        $$
        We also consider the variant $\rflip(\varphi) = \flip_2 \circ \varphi$ and $\sflip(\varphi) = \varphi \circ \flip_1$.
\end{deff}
\begin{example}
  Consider the double tangent bundle $TTM$ and let $\pi^h,\, \pi^v$ be its horizontal and vertical projection, respectively. 
  Recall from $\S \ref{par.dvb-lift}$ that $\pi^h$ lift to a dvb-map 
  $\underline{\pi}^h : TTM \rightarrow TM^h$ given by 
  $\left[\begin{smallmatrix}
            \pi^h & \id 
            \\ \pi & \id 
    \end{smallmatrix}\right]$.
    Then 
    $ \flip(\underline{\pi}^h) : \flip(TTM) \rightarrow TM^v$
    is given by
    $\left[\begin{smallmatrix}
            \pi^h & \pi 
            \\  \id& \id 
    \end{smallmatrix}\right]$.
    In the same way, let $\underline{\pi}^v$ be the dvb-map $TTM \rightarrow TM^v$ given by 
    $\left[\begin{smallmatrix}
            \pi^v & \pi 
            \\  \id & \id 
    \end{smallmatrix}\right]$, 
    then 
    $\flip(\underline{\pi}^v) = 
        \left[\begin{smallmatrix}
        \pi^v &  \id
        \\  \pi & \id 
    \end{smallmatrix}\right]$ 
    defines a dvb-map $\flip(TTM) \rightarrow TM^h$.
\end{example}
\begin{deff}\label{def.flip-pb}
Let $\Upsilon: (D_1, F, B_1, M) \flipar (D_2, A_2, F, M)$ be a flip map, and $\psi: (E,N) \rightarrow (F,M)$ be a vb-map. 
Then, the \textbf{pullback} of $\Upsilon$ is the flip map
$$
  \psi^{*, \htov}\Upsilon: \psi^{*,h} D_1 \longflipar \psi^{*,v}D_2
$$
obtained as the restriction of $\flip \times \Upsilon: E^h \times D_1 \flipar E^v \times D_2$. In the same manner, one defines a pullback flip map from the horizontal pullback to the vertical pullback.
\end{deff}
\begin{example}\label{Phi-as-pb}
Consider the identity flip $\id: TTN \flipar TTN$. Then, the flip isomorphism $\Phi$ from~(\ref{Phlip}) is essentially  the pullback flip $$f^{*, \htov} \id: (f_*)^{*,h} TTN \flipar (f_*)^{*,v} TTN,$$ under the natural isomorphism $(f_*)^{*,v} TTN \cong T(f^*TN)$.
\end{example}

%
%
%
\subsection{Quotients as cokernel}
%
%
%
\subsubsection{Quotient by wide double vector bundle.}
In this work, the quotients of double vector bundle by (wide) double vector sub-bundle will be treated as quotients of some kind of internal objects in the category of vector bundles. 
Roughly speaking, the idea is to consider a subcategory of $\DVB$ which behaves like $\VB$, in which a suitable class of morphisms have well-defined kernels and cokernels. 
To do so, simply fix a vector bundle $E$, and consider all the double vector bundles having $E$ as, say horizontal, side-bundle. 
We are opting for a formal definition in terms of homogeneity structure:
\begin{deff}\label{def.vb-object}
    Let $(E,\delta,M)$ be a smooth vector bundle.
    A \textbf{vb-object} over $E\! \rightarrow\! M$ (or ``vb-vector bundle'') consists in a tuple $(D, \kappa, j,\tau)$ where:
    \begin{itemize}
    \setlength\itemsep{0pc}
     \item[(i)] $(D,\kappa)$ is a smooth vector bundle with base $B:= \kappa(0,D)$ and homogeneity action $\kappa: \R \times D \rightarrow D$,
     \item[(ii)] $j: (E,\delta, M) \hookrightarrow (D, \kappa, B)$  is a vb-embedding (in particular, $j^*\kappa =\delta$), 
     \item[(iii)] $\tau: \R \times D \rightarrow D$ is a regular $(\R,\cdot)$-action, called \textbf{structural homogeneity structure}, such that $\tau$ acts by vector bundle endomorphisms (or equivalently, the actions $\kappa$ and $\tau$ commute, $\kappa_\lambda \circ \tau_\mu = \tau_\mu \circ \kappa_\lambda$)
     \end{itemize}
     satisfying the following compatibility conditions:  
     \begin{itemize}
     \setlength\itemsep{0pc}
     \item[(v)] $j^*\tau = \id_E$ (the image of $j$ is fixed by $\tau$), and
     \item[(iv)] $\tau(\lambda, d) = d$ implies $d \in j(E)$, and $\tau(\lambda,b)= b$ implies $b \in j(M)$ (the $\tau$-fixed point are in the image of $j$).
    \end{itemize}
    Such vb-object will be denoted $[D,B] \rightarrow [E,M]$.
\end{deff}
\begin{remark}
    The condition (iii) in the definition above implies that the $\tau$-fixed point define a smooth vector sub-bundle $(\tau_0(D), \kappa\vert_{\tau_0(D)}, \tau_0(B))$ of $(D, \kappa, B)$,   
    The condition $(v)$ and $(iv)$ guarantee that this latter sub-bundle is exactly  the image of the vb-embedding $j$.
\end{remark}
\begin{deff}
    A \textbf{morphism of vb-object} over $E \rightarrow M$
    $$
    \varphi: (D_1,\kappa_1, j_1, \tau_1) \rightarrow (D_2,\kappa_2, j_2, \tau_2)
    $$
    is a vb-map $\varphi: (D_1, \kappa_1) \rightarrow (D_2, \kappa_2)$ compatible with the vb-object structure in the sense that:
        $\tau_2 \circ (\id \times \varphi) = \varphi \circ \tau_1$ and $\varphi \circ j_1 = j_2$.
    The vb-object over $E\rightarrow M$ together with the morphisms of vb-object form a category denoted $\VB(E\!\rightarrow\! M)$ with zero object $[E, M] \rightarrow [E,M]$.
    An \textbf{embedding of vb-object} over $E \rightarrow M$ is a morphism of vb-object such that $\varphi: (D_1,\kappa_1) \rightarrow (D_2,\kappa_2)$ is a vb-embedding. 
    The notion of isomorphism is defined similarly.
\end{deff}
\begin{prop}
 Let $(D,\kappa,j,\tau)$ be a vb-object over $\pi:E\rightarrow M$. Then $D$ admits a natural double vector bundle structure (determined up to flip).
\end{prop}
\begin{proof}
Consider the fiber product 
 $$
 \begin{array}{rcl}
   D \!\underset{\tau_0(D)}{\times} \! E & = & \{ (d,e) \in D \times E  \: : \: \tau_0(d)= j(e) \}, 
   \\ 
   B \!\underset{\tau_0(B)}{\times} \! M & = & \{ (b,m) \in B \times M  \: : \: \tau_0(b)= j(m) \}.
 \end{array}
 $$
 Then the following square defines a double vector bundle:
 \begin{align*}
  \xymatrix@C=3pc@R=1.5pc{
    D \!\underset{\tau_0(D)}{\times} \! E 
    \ar[r]^-{\pr_2} \ar[d]_-{\kappa_0 \times \pi}
    & E \ar[d]^-{\pi}
    \\
    B  \!\underset{\tau_0(B)}{\times} \! M \ar[r]_-{\pr_2} & M
 }
 \end{align*}
 Indeed, the horizontal homogeneity action on  $D \times_{\tau_0(D)}  E$ is induced by the product action $\tau \times \id$, whereas the vertical one is induced by $\kappa \times \delta$, which clearly commute with each other. 
 Next, since $j$ realizes a diffeomorphism $E \cong \tau_0(D)$, it gives a natural diffeomorphism $D \times_{\tau_0(D)}E \cong D$, and similarily  $B \times_{\tau_0(B)}M  \cong M$, endowing in this way $D$ with a double vector structure.
\end{proof}
\begin{remark}
    Conversely, every double vector bundle $(D,A,B,M)$ is a vb-object, but in two distinct ways: either horizontally as
    $\left[\begin{smallmatrix}
     D \\ B
    \end{smallmatrix}\right]
    \rightarrow 
    \left[\begin{smallmatrix}
     E \\ M
    \end{smallmatrix}\right]$, or vertically as $\left[\begin{smallmatrix}
     D & A
    \end{smallmatrix}\right] \rightarrow 
    \left[\begin{smallmatrix}
     B & M
    \end{smallmatrix}\right]$. For the sake of efficiency, the definitions will often be stated only for the horizontal vb-objects, the vertical counterpart being completely analogous (under flip~\S\ref{par.dbcat-flip}).
\end{remark}
From this perspective, given a horizontally wide dvb-embedding\footnote{The discussion can be adapted for any horizontally wide constant-rank dvb-map $\varphi$.} $\varphi: Q \hookrightarrow D$, there is an associated canonical quotient projection $\quot_h$ fitting in a sequence of dvb-map:
\begin{align*}
    \begin{split} 
        \left[
        \vcenter{
        \xymatrix@R=0pc@C=0pc{
            Q & E \\ B  & M
        }}    
        \right]
        \xrightarrow{ \quad 
        \left[\begin{smallmatrix}
         \varphi & \id 
        \\ \psi & \id 
        \end{smallmatrix}\right]
        \quad
        }
        \left[
        \vcenter{
        \xymatrix@R=0pc@C=0pc{
            D 
            & E 
            \\ A 
            & M
        } }
        \right]
        \xrightarrow{ \quad 
        \quot_h
        \quad
        }
        \left[
        \vcenter{
        \xymatrix@R=0pc@C=0pc{
            D/Q 
            & E 
            \\ A/B 
            & M
        } }
        \right]
    \end{split}
\end{align*}
 which we interpret as a short exact sequence of horizontal vb-objects over $E \rightarrow M$, where $\varphi$ is now viewed as an embedding of vb-objects,
    \begin{align*}
    \begin{split} 
        0\xrightarrow{\qquad}
        \left[
        \vcenter{
        \xymatrix@R=0pc@C=1pc{
            Q \\ B  
        }}   
        \right]
    \xrightarrow{ 
        \quad \varphi \quad
        }
         \left[
        \vcenter{
        \xymatrix@R=0pc@C=1pc{
            D \\ A  
        }}   
        \right]
        \xrightarrow{ 
        \quad \quot_h  \quad
        }
         \left[
        \vcenter{
        \xymatrix@R=0pc@C=1pc{
            D/Q \\ A/B  
        }}   
        \right]
        \xrightarrow{\qquad} 0
    \end{split}
\end{align*}
The category of horizontal vb-objects over $E \rightarrow M$, together with their morphisms, will be denoted $\VB^h(E \!\rightarrow \!M)$.
The zero-object in $\VB^h(E \!\rightarrow \!M)$ corresponds to the double vector bundle $E^h$ (see~\S\ref{par.dvb-lift}), and the morphisms are precisely the class of dvb-maps of the shape 
$ 
\left[
        \begin{smallmatrix}
         * & \id 
        \\ * & \id 
        \end{smallmatrix}
\right] 
$.
Moreover, as in the category $\VB$, the kernels and cokernels make sense for any constant-rank map (that is, every vb-map involved is of constant rank).
The category $\VB^v(E\!\rightarrow\! M)$ of vertical vb-objects over $E\rightarrow M$ shares the same kind of properties, with quotient projection denoted by $\quot_v$.
\begin{deff}
Let $\varphi: Q \hookrightarrow D$ be an embedding in $\VB^h(E\! \rightarrow\! M)$ as above. 
 The \textbf{horizontal quotient} $D/_{\! h}Q$ 
 (occasionally denoted as $\frac{D}{Q}\!{\scriptstyle h}$) 
 is the double vector bundle with total space $D/Q$ (as quotient vector bundle over $E$) equipped with the following quotient homogeneity structure:
$$
\begin{array}{rclcl}
     \delta^{D/Q}_{E}(\lambda, [d]) 
     &=&
     [\delta^D_E(\lambda, d)] &=& (\delta^D_E / \delta^Q_E) (\lambda,[d]),
     \\[0.5pc]
     \delta^{D/Q}_{A/B}(\mu, [d])
     &=&
     [\delta^D_B(\mu, d)] & = & (\delta^{D}_B / \delta^Q_A)(\mu,[d]).
\end{array}
$$
\end{deff}
\begin{prop}
 $(D/_{\! h}Q, \quot_h)$ is a cokernel for $\varphi$ in $\VB^h(E \rightarrow M)$.
\end{prop}
\begin{remark}
  In the same manner, the cokernel of an embedding of vertical vb-object in $\VB^v(E\rightarrow M)$, 
  $$
  \left[
        \vcenter{
        \xymatrix@R=0pc@C=0pc{
            P 
            & A  
        } }
        \right]
        \xymatrix{\xhookr[1.5pc]^-\Psi &}
        \left[
        \vcenter{
        \xymatrix@R=0pc@C=0pc{
            D 
            & B 
        } }
        \right]
    $$
  is called the \textbf{vertical quotient}, denoted by $\frac{D}{P}{\! \scriptstyle v}$.
   It corresponds to the double vector bundle with underlying total space $D/P$ and homogeneity structure given by
  $$
        \delta^{D/P}_{B/A}(\lambda, [d]) 
        =(\delta^D_B/\delta^D_A)(\lambda, [d]), \qquad 
        \delta^{D/P}_{E} (\mu, [d]) = (\delta^D_E/\delta^P_E)(\mu,[d]).
$$
\end{remark}
\begin{remark}
 Notice the difference with the comments of~\cite[\S 4.5]{meinrenken2022dvbquot}, the kernel of the quotient projection $D \rightarrow D/_{\! \scriptstyle h}Q$ is $Q$ when properly interpreted in the category of vb-object over $E \!\rightarrow \! M$ (instead of the category of double vector bundles and dvb-maps).
\end{remark}
\subsubsection{Quotient and tangent functor.}
\begin{prop}\label{prop.quot-tgt}
Let $F \hookrightarrow E$ be a wide vb-embedding.
Then, there is natural dvb-isomorphism
 $
 \frac{TE}{TF}{ \!\scriptstyle v} \cong T(E/F)
 $
 where $TE,TF$ and the right-hand side are endowed with the dvb-structure of the example~\ref{ex.dvb-TE}.
 Naturality goes as follows: 
 any 2-map $(\varphi,\psi): j_1 \Rightarrow j_2$ in $\Embed\VB$ induces a canonical identification of dvb-maps
 $$
 \left. \vcenter{
 \xymatrix@R=1.5pc@C=1pc{
    \frac{TE_1}{TF_1}{\!\scriptscriptstyle v} 
    \ar[d]_{
      \frac{\varphi_*}{\psi_*}{\!\scriptscriptstyle v}
    } 
    \ar@{}[r]|-{\rotatebox[origin=c]{0}{$\cong$}}& T(E_1/F_1) \ar[d]^-{(\varphi/\psi)_*} 
    \\
    \frac{TE_2}{TF_2}{\!\scriptscriptstyle v} \ar@{}[r]|-{\rotatebox[origin=c]{0}{$\cong$}}& T(E_2/F_2)
 }
 }\right.
 $$
\end{prop}
\begin{proof}
   The dvb-isomorphism $\frac{TE}{TF}{ \!\scriptstyle v} \cong T(E/F)$ follows from the universality of the cokernels and the fact that the tangent functor preserves exactness. 
   Precisely, 
   $
   0 \rightarrow F \xrightarrow{j} E \xrightarrow{q} E/F \rightarrow 0
   $
   is a short exact sequence of vector bundle over a fixed base manifold $M$. Then, applying the tangent functor yields the short exact sequence 
   $$
   0 \longrightarrow 
   \left[\begin{matrix}TF & F\end{matrix}\right] 
   \xrightarrow{ [j_* \:\, j] }\left[\begin{matrix}TE & E\end{matrix}\right] 
   \xrightarrow{[q_* \:\, q]} 
   \left[\begin{matrix}T(E/F) & E/F\end{matrix}\right]
   \longrightarrow 0
   $$ 
   in the category $\VB^v(TM \rightarrow M)$. Consequently, the vertical quotient $\frac{TE}{TF}{ \!\scriptstyle v}$ associated to $[j_* \: \, j]$ is isomorphic (in a unique way) to $\left[\begin{matrix}T(E/F) & E/F\end{matrix}\right]$ in $\VB^v(TM \rightarrow M)$, thus also as double vector bundles. 
   Next, consider a morphism of short exact sequences of vector bundles, or equivalently a short exact sequence of 2-vb-maps\footnote{Notice that throughout the proof, we took the liberty to deal with 2-maps in $\Immer\VB^h$ (instead of $\Immer\VB$, which stands for $\Immer\VB^v$ under our conventions) for the sake of saving some space. Fortunately, there is no loss of generality.}, 
   $$
   \xymatrix@R=1.5pc@C=3pc{
   0 \ar[r] & F_1 \ar[r]^{j_1} \ar[d]_-{\psi} \ar@{}[rd]|{\rotatebox[origin=c]{0}{$\overset{J}{\Rightarrow}$}} & E_1 \ar[r]^-{q_1} \ar[d]_-{\varphi} \ar@{}[rd]|{\rotatebox[origin=c]{0}{$\overset{Q}{\Rightarrow}$}} & E_1/ F_1 \ar[r] \ar[d]^-{\varphi / \psi} & 0
   \\
   0 \ar[r] & F_2 \ar[r]_{j_2}  & E_2 \ar[r]_-{q_2} & E_2/ F_2 \ar[r]& 0
   }
   $$
   where the first row is over $M_1$ and the second one is over $M_2$. 
   One see that $Q: \varphi \Rightarrow \varphi/\psi$ is a cokernel (in a suitable category) for the 2-vb-embedding $J: \psi \Rightarrow \varphi$.
   Applying the tangent functor to the quotient vb-map $\varphi/\psi$ yields the dvb-map 
   $$
   \left[ \vcenter{
   \xymatrix@R=0pc@C=0pc{
   T(E_1 / F_1) & E_1/ F_1  
   \\
   TM_1 & M_1 
   } 
   } \right]
   \xrightarrow{
   \quad 
   \left[ 
   \begin{smallmatrix}
    (\varphi/\psi)_* & \varphi/\psi \\ f_* & f
   \end{smallmatrix}
   \right]
   \quad
   }
   \left[ \vcenter{
   \xymatrix@R=0pc@C=0pc{
   T(E_2 / F_2) & E_2/ F_2 
   \\
   TM_2 & M_2 
   } 
   } \right]
   $$
   in such a way that $Q_*: \varphi_* \Rightarrow (\varphi/\psi)_*$ is a cokernel for $J_*: \psi_* \Rightarrow \varphi_*$ in the subcategory of morphisms of the horizontalization of the double category of $(\Immer\DVB, \DVB)^\square$, whose 2-morphisms are (horizontally directed) squares of dvb-maps of the shape: an upper arrow in $\VB^v(TM_1 \rightarrow M_1)$, a lower arrow in $\VB^v(TM_2 \rightarrow M_2)$, and the lateral arrows are dvb-maps with vertical side vb-map $f_*: TM_1 \rightarrow TM_2$. But, since this latter category admits a terminal object, namely the vertical lift $f_*^v: TM_1^v \rightarrow TM_2^v$, there is a unique compatible isomorphism $(\varphi/\psi)_* \cong \frac{\varphi_*}{\psi_*}{\!\scriptscriptstyle v}$.
\end{proof}

\subsubsection{Quotient and side-pullbacks.}
\begin{prop}\label{prop.vpb-hquot}
    The vertical side-pullback is compatible with the horizontal quotient in the sense that: if $(Q,E,B,M) \hookrightarrow (D,E,A,M)$ is a dvb-embedding defining  a embedding in $[Q, B] \hookrightarrow [D,A]$ in $\VB^h(E\!\rightarrow \! M)$, and $\Upsilon: (V\hookrightarrow U) \Rightarrow (B \hookrightarrow A)$ is a 2-map of wide vb-embeddings given by the pair $(\varphi,\psi)$, then there is a canonical dvb-isomorphism
    \begin{align}\label{eq.vpb-hquot}
     \begin{split}
        \left(\varphi / \psi \right)^{*,v} 
        \left(D \big/_{\!\!{\scriptstyle h}} Q \right)  
        \cong
        (\varphi^{*,v}D)  \big/_{\!\!{\scriptstyle h}} (\psi^{*,v} Q).
     \end{split}
    \end{align}
\end{prop}
\begin{remark}
 Conversely, the horizontal side-pullback is compatible with the vertical quotient.
\end{remark}
\begin{proof}
 On the left hand-side, the horizontal quotient $D/_{\! h} Q$ fits into the following short exact sequence in $\VB^h(E\!\rightarrow\! M)$
\begin{align*}
    \begin{split} 
    \xymatrix{
        0 \ar[r] 
        & 
        {\left[
        \begin{matrix}
         Q  \\ B 
        \end{matrix}
        \right]}
        \xhookr[1.4pc]^-{\iota} 
        &
        {\left[
        \begin{matrix}
            D  \\ A 
        \end{matrix}
        \right]}
        \ar[r]^-{\quot_h}
        &
        {\left[
        \begin{matrix}
            D/Q \\ A/B 
        \end{matrix}
        \right]}
        \ar[r]
        &
        0.
    }
    \end{split}
\end{align*}
Let $f$ be the base-map of $\varphi$ and $\psi$.
Let $\Theta: (U \rightarrow A) \Rightarrow (U/V \rightarrow A/B)$ be the 2-vb-map given by the pair $(\varphi/\psi, \varphi)$. Then, by performing the vertical pullbacks of $\iota$ and $\quot_h$ by $\Upsilon$ and $\Theta$ respectively, we obtain the following short exact sequence in $\VB^h(f^*E\! \rightarrow \! M)$,
\begin{align*}
    \begin{split} 
    \xymatrix@C=3pc{
        0 \ar[r] 
        & 
        {\left[
        \begin{matrix}
         \psi^*Q  \\ V 
        \end{matrix}
        \right]}
        \xhookr[2.4pc]^-{\Upsilon^*\iota} 
        &
        {\left[
        \begin{matrix}
            \varphi^*D  \\ U 
        \end{matrix}
        \right]}
        \ar[r]^-{\Theta^*\quot_h}
        &
        {\left[
        \begin{matrix}
            (\varphi/\psi)^*(D/Q) \\ U/V 
        \end{matrix}
        \right]}
        \ar[r]
        &
        0.
    }
    \end{split}
\end{align*}
That is, $(\varphi/\psi)^{*,v}(D/_{\! h}Q)$ is a cokernel for $\Upsilon^*\iota$. But, since $(\varphi^*D)/_{\! h}(\psi^*Q)$ is also a cokernel for $\Upsilon^*\iota$, there exists an isomorphism $(\varphi/\psi)^{*,v}(D/_{\! h}Q) \cong (\varphi^*D)/_{\! h}(\psi^*Q)$ as objects in the category $\VB^h(f^*E \! \rightarrow \! M)$, and thus as double vector bundles. Moreover, the latter isomorphism is uniquely determined by its compatibility with the respective legs\footnote{Using the terminology from~\cite{riehl-context}.}.
\end{proof}
\subsubsection{Quotient of dvb-maps.} 
Let $\iota_1: Q_1 \hookrightarrow D_1$ and $\iota_2: Q_2 \hookrightarrow D_2$ be two embeddings in $\VB^v(E_1 \!\rightarrow \!M_1)$ and $\VB^v(E_2 \!\rightarrow\! M_2)$, respectively. Given a 2-dvb-map $(\varphi,\psi) : \iota_1 \Rightarrow \iota_2$ such that the vertical side-maps of $\varphi$ and $\psi$ coincide (say, a ``base-change'' for vb-objects), there is an associated dvb-map $\frac{\varphi}{\psi}{\! \scriptstyle v}$, called the \textbf{vertical quotient} of $\varphi$ by $\psi$, uniquely characterized by the commutativity of the following diagram:
$$
\xymatrix@C=3.5pc{
  {\left[
    \begin{matrix}
      Q_1 & B_1\\ E_1 & M_1          
    \end{matrix}
   \right]}
  \ar[d]_-{
  {\left[
   \begin{smallmatrix}
    \psi &  \beta \\ \varepsilon & f          
   \end{smallmatrix}
   \right]}
  }
   \xhookr[3pc]^-{
   {\left[
    \begin{smallmatrix}
      \iota_1 &  i_1\\ \id & \id          
    \end{smallmatrix}
    \right]
   }}
  &
  {\left[ \begin{matrix}
        D_1  & A_1\\ E_1 & M_1          
    \end{matrix}\right]}
   \ar[d]^-{
   {\left[ \begin{smallmatrix}
    \varphi &  \alpha \\ \varepsilon & f         
   \end{smallmatrix}\right]}} 
   \ar[r]^-{
   {\left[
    \begin{smallmatrix}
      q &  q \\ \id & \id          
    \end{smallmatrix}
    \right]
   }} 
   &
   {\left[ \begin{matrix}
    \frac{D_1}{Q_1}{\! \scriptstyle v} & \frac{A_1}{B_1} \\ E_1 & M_1          
  \end{matrix}\right]}
  \ar[d]^-{
   {\left[
    \begin{smallmatrix}
      \frac{\varphi}{\psi}{\! \scriptscriptstyle v} &  \frac{\alpha}{\beta} \\ \varepsilon & f          
    \end{smallmatrix}
    \right]}}
    \ar[r]
    & 0
   \\
   {\left[
    \begin{matrix}
      Q_2 & B_2\\ E_2 & M_2          
    \end{matrix}
   \right]}
    \xhookr[3pc]_-{
   {\left[
    \begin{smallmatrix}
      \iota_2 &  i_2\\ \id & \id          
    \end{smallmatrix}
    \right]
   }} 
  &
  {\left[ \begin{matrix}
    D_2 & A_2 \\ E_2 & M_2          
  \end{matrix}\right]}
  \ar[r]_-{
   {\left[
    \begin{smallmatrix}
      q &  q \\ \id & \id          
    \end{smallmatrix}
    \right]}} 
    & 
    {\left[ \begin{matrix}
    \frac{D_2}{Q_2}{\! \scriptstyle v} & \frac{A_2}{B_2}\\ E_2 & M_2          
  \end{matrix}\right]}
  \ar[r]
    & 0
 }
 $$
 where $q$ stands for the quotient projections and the 0's denote the zero vb-objects in the corresponding category.
  As in remark~\ref{rem.hor-cokernel}, the vertical quotient dvb-map may be interpreted as a cokernel in a suitable category of 2-dvb-map. More generally, the above definition still holds when $\iota_1,\iota_2$  are replaced by dvb-maps of constant rank. 
  The horizontal quotient of dvb-maps is defined along the same lines.
%
%
%
\subsubsection{Quotient flip maps.}\label{par.quotient-flip}
Let $\iota : Q \hookrightarrow D$ be an embedding of vertical vb-object over $E \! \rightarrow \! M$. 
Then, the associated vertical quotient $\frac{D}{Q}{\! \scriptstyle v}$ belongs to $\VB^v(E \!\rightarrow \!M)$. 
Likewise, since $\flip(\iota)$ is an embedding, the horizontal quotient $D /_{\! h} Q$ makes sense, but this time as an object in the category $\VB^h(E\!\rightarrow\! M)$. Let $\flip_Q$ and $\flip_D$ be the respective total flips, then their quotient makes sense as the unique map $\flip_D / \flip_Q$ fitting in the following commutative diagram:
$$
\xymatrix@C=2pc@R=1.6pc{
  {\left[ \begin{matrix}
    Q & B          
   \end{matrix}\right]}
   \flipard_-{\flip_Q}
   \xhookr[1.4pc]^-{\iota }
   &
   {\left[ \begin{matrix}
    D & A        
   \end{matrix}\right]}
   \flipard^-{\flip_D} 
   \ar[r]^-{\quot_v} 
   &
   {\left[ \begin{matrix}
     \frac{D}{Q}{\! \scriptstyle v}  &  \frac{A}{B} 
   \end{matrix}\right]}
   \flipard^-{\flip_D/\flip_Q }
    \ar[r]
    & 0
   \\
   {\left[ \begin{matrix}
    Q  \\ B          
   \end{matrix}\right]}
    \xhookr[2.6pc]_-{\flip(\iota)} 
   &
   {\left[ \begin{matrix}
    D \\ A        
   \end{matrix}\right]}
   \ar[r]_-{\quot_h} 
    & 
    {\left[ \begin{matrix}
    D /_{\! h} Q\\  A /B        
  \end{matrix}\right]}
  \ar[r]
    & 0
 }
 $$
 where $\quot_h$ and $\quot_v$ denote respectively the horizontal and vertical quotient projections.
 In particular, by a direct comparison of the homogeneity structure, one finds that $D /_{\! h} Q = \flip (\frac{D}{Q}{\! \scriptstyle v} )$ and $\flip_D/\flip_Q = \flip_{D/Q}$.
 More generally, the definition of the quotient flip also makes sense for general flip isomorphisms (definition~\ref{def.flip-map}) instead of the total flips.
\begin{remark}\label{rem.quot-deco}
    Observe that the family of vertical arrows in the diagram above define a morphism from a sequence of vertical vb-object to a sequence of horizontal vb-objects (or conversely). 
    For the sake of rigorousness, one would decorate the quotient flip as $\flip_D {}^{\scriptstyle h \!\!}/_{\! \scriptstyle v} \,\flip_Q$ in order to explicit the corresponding categories (conversely $\flip_D {}^{\scriptstyle v \!\!}/_{\! \scriptstyle h} \,\flip_Q$). 
    Notice that the identity flip of a symmetric double vector bundle $(D,E,E,M)$ may be interpreted in the following way: it sends $D$, as an object of $\VB^h(E\!\rightarrow\! M)$, to $D$, as an object of $\VB^v(E\!\rightarrow\! M)$.
\end{remark}
 %
 %
\begin{prop}\label{prop.quot-flip}
  Consider two vector bundles $E_1 \rightarrow M_1$ and $E_2 \rightarrow M_2$.
  Let $\iota_1: P_1 \hookrightarrow D_1$ be an embedding in $\VB^v(E_1\! \rightarrow \! M_1)$, and $\iota_2: P_2 \hookrightarrow D_2$ be an embedding in $\VB^h(E_2 \!\rightarrow\! M_2)$. Assume that $(\varphi, \psi)$ defines a 2-dvb-map $\flip(\iota_1) \Rightarrow \iota_2$ such that $\varphi$ and $\phi$ have the same horizontal side-maps. 
  Then 
  $$
  (\varphi /_{\! \scriptstyle h} \, \psi) \circ (\flip_D / \flip_P) = (\varphi \circ \flip_D) / (\psi \circ \flip_P).
  $$
  Moreover, the same holds for arbitrary flip maps instead of the total flip.
\end{prop}
 \begin{proof}
 Consider the concatenation of the given data, in the horizontalization of $\DVB_f^\square$ (the double category of commutative square in $\DVB_f$)
  $$
\xymatrix@C=2.5pc@R=1.8pc{
  {\left[ \begin{matrix}
    P_1 & B_1 \\ E_1 & M_1          
   \end{matrix}\right]}
   \flipard_-{\flip_Q}
   \ar[r]^-{\iota_1 }
   \ar@{}[rd]|{
    \rotatebox[origin]{0}{$\overset{I_1}{\Rightarrow}$}
    }
   &
   {\left[ \begin{matrix}
    D_1 & A_1 \\ E_1 & M_1          
   \end{matrix}\right]}
   \flipard^-{\flip_D} 
   \ar[r]^-{q_1} 
   \ar@{}[rd]|{
    \rotatebox[origin]{0}{$\overset{Q_1}{\Rightarrow}$}
    }
   &
   {\left[ \begin{matrix}
     \frac{D_1}{P_1}{\! \scriptstyle v}  &  \frac{A_1}{B_1} \\ E_1 & M_1          
   \end{matrix}\right]}
   \ar[d]^-{\flip_D/\flip_P }
    \ar[r]
    & 0
   \\
   {\left[ \begin{matrix}
    P_1 & E_1 \\ B_1 & M_1          
    \end{matrix}\right]}
    \ar[r]_-{\flip(\iota_1)} 
    \ar[d]_-{\psi}
    \ar@{}[rd]|{
    \rotatebox[origin]{0}{$\overset{I_2}{\Rightarrow}$}
    }
    &
    {\left[ \begin{matrix}
    D_1 & E_1 \\ A_1 & M_1          
    \end{matrix}\right]}
    \ar[d]^-{\varphi}
    \ar[r]_-{\flip(q_1)} 
    \ar@{}[rd]|{
    \rotatebox[origin]{0}{$\overset{Q_2}{\Rightarrow}$}
    }
    &
    {\left[ \begin{matrix}
    D_1 /_{\! h} \, P_1 & E_1 \\ A_1 /B_1 & M_1          
    \end{matrix}\right]}
    \ar[r]
    \ar[d]^-{\varphi /_{\!\scriptscriptstyle  h} \psi}
    & 0
     \\
   {\left[
    \begin{matrix}
      P_2 & E_2\\  B_2& M_2          
    \end{matrix}
   \right]}
    \ar[r]_-{\iota_2}          
  &
  {\left[ \begin{matrix}
    D_2 & E_2 \\  A_2 & M_2          
  \end{matrix}\right]}
  \ar[r]_-{q_2} 
    & 
    {\left[ \begin{matrix}
     D_2/_{\! \scriptstyle h}\, P_2 &E_2 \\  A_2 / B_2& M_2          
  \end{matrix}\right]}
  \ar[r]
    & 0
 }
 $$
By the interchange law,
$
  \left(
    Q_2 \bullet Q_1
  \right) \circ 
  \left( 
     I_2 \bullet I_1
  \right) =
   \left(
    Q_2 \circ I_2 
  \right)
  \bullet
  \left(
    Q_1 \circ I_1
  \right).
$
 In particular, $ Q_2 \bullet Q_1: \varphi \circ \flip_D \Rightarrow  (\varphi/_{\! \scriptstyle h} \, \psi) \circ  (\flip_D / \flip_P) $ is a cokernel for the vertical composition $I_2 \bullet I_1: \psi \circ \flip_P \Rightarrow \varphi \circ \flip_D$. 
 Thus, by universality of the cokernel, there is a unique isomorphism of (flipping) dvb-map
 $$
 \xymatrix@C=6pc{
    \displaystyle\frac{D_1}{P_1}{\! \scriptstyle v} 
    \ar@/^1pc/[r]^{(\varphi \, \circ \,\flip_D) / (\psi \,\circ \,\flip_P)} \ar@/^1pc/[r]|{\rule{0.8pt}{5pt}}
    \ar@{}[r]|{\rotatebox[origin=c]{-90}{$\Rightarrow$}}
    \ar@/_1pc/[r]_{(\varphi/_{\! \scriptscriptstyle h}\, \psi)  \circ (\flip_D/ \flip_P)}\ar@/_1pc/[r]|{\rule{0.8pt}{5pt}}
    &
    D_2/_{\! \scriptstyle h}\, P_2
 }
 $$
 But since $Q_2 \bullet Q_1 =  (q_2, q_1)$, the isomorphism between these cokernels should be the identity.
 The same argument works for the case of general flip maps.
\end{proof}
\begin{remark}
Let $\varepsilon: E_1 \rightarrow E_2$ be a vb-map.
Denote by $\VB_v^h(\varepsilon)$ the category whose objects are compositions $\varphi \circ \Theta$ of the shape
$$
\xymatrix{
  \VB^v(E_1 \!\rightarrow \!M_1)
  \ar@{..>}[r]^{\Theta}
  & \VB^h(E_1 \!\rightarrow \!M_1) 
  \ar@{..>}[r]^{\varphi}
  &\VB^h(E_2 \!\rightarrow\! M_2)
  }
$$
where $\Theta$ ranges over flip isomorphisms between vb-objects over $E_1 \!\rightarrow\! M_1$, $\varphi$ ranges over the morphism of horizontal vb-objects projecting to $\varepsilon$, and whose morphisms are given by morphisms of such sequences.
Then, the zero object in $\VB^h_v(\varepsilon)$ is precisely the composition 
$
E_1^v \overset{\flip}{\flipar} E_1^h \xrightarrow{\varepsilon^h} E_2^h.
$
\end{remark}
\subsection{Normal bundle of a vb-immersion}
Let $\varphi : E \looparrowright F$ be a vb-immersion over $j: M \looparrowright N$ (cf.~definition~\ref{def.vbimmersion}).
Then $\varphi_*: TE \rightarrow TF$ is a dvb-immersion, and the horizontal sharpening $\varphi_\sharp:  TE \rightarrow \varphi^{*,h} TF$ defines an embedding in the category $\VB^h(E\! \rightarrow \! M)$.
\begin{deff}\label{def.vbnormal}
    The \textbf{vb-normal bundle} $\nuup_\VB(\varphi)$ is defined as the following quotient in the category of $\VB^h(E\!\rightarrow\! M)$ (where $\quot_h$ is the quotient projection):
\begin{align*}
    \begin{split} 
        0 \xrightarrow{\hspace{1pc}}
        \left[
        \vcenter{
        \xymatrix@R=0pc@C=0.2pc{
            TE   \\ TM 
        }}    
        \right]
        \xrightarrow{ \quad 
        \left[\begin{smallmatrix}
         \varphi_\sharp  
        \\ j_\sharp  
        \end{smallmatrix}\right]
        \quad
        }
        \left[
        \vcenter{
        \xymatrix@R=0pc@C=0.2pc{
            \varphi^*TF              
            \\ j^*TN 
        } }
        \right] 
        \xrightarrow{\quad \quot_h \quad} 
        \left[
        \vcenter{
        \xymatrix@R=0pc@C=0.2pc{
            \nuup(\varphi)              
            \\ \nuup(j) 
        } }
        \right]
        \xrightarrow{\hspace{1pc}}
        0
        .
    \end{split}
\end{align*}
\end{deff}
The homogeneity structure of $\nuup_\VB(\varphi)$ is given by: for $\lambda \in \R$, $e\in E$ and $\eta \in T_{\varphi(e)}F$,
$$
\begin{array}{rcl}
    \delta^{\nuup(\varphi)}_{E} (\lambda, [e, \eta])
    &=&
    [e, \delta^{TF}_F(\lambda, \eta)]
    \\[0.4pc]
    \delta^{\nuup(\varphi)}_{\nuup(j)} (\lambda, [e, \eta])
    &=& 
    [\delta^E(\lambda, e) , \delta^{TF}_{TN}(\lambda, \eta)]
\end{array}
$$
Viewed as a commutative square of vector bundles, the vb-normal bundle writes as
\begin{align}\label{diag.vbnormal}
    \begin{split}
        \xymatrix@C=2pc@R=1.5pc{
        \nuup(\varphi)
        \ar[r]^-{\pi^{\nuup(\varphi)}}
        \ar[d]_-{\nuup^{(2)}(\underline{\smash \pi})}
        & E \ar[d]
        \\
        \nuup(j) 
        \ar[r]
        & M
        }
    \end{split}
\end{align}
The horizontal  projection $\nuup(\varphi) \rightarrow E$ is the usual projection of the normal bundle. 
The vertical projection  $\nuup(\varphi) \rightarrow \nuup(j)$ is the normal differential $\nuup^{(2)}(\underline{\smash \pi})$ where $\underline{\smash \pi}: \varphi \Rightarrow j$ is given the pair of bundle projection $(\pi_F, \pi_E)$.
In summary, the definition above provide us with an operation
$$
\begin{array}{rcl}
     \nuup_\VB: \quad \{ \text{vb-immersions} \} & \mapsto&  \{ \text{double vector bundles} \}.
\end{array}
$$
Usually, we just write $\nuup$ instead of $\nuup_{\VB}$.
\subsubsection{Example: the tangent normal bundle.}
 Let $j: M \rightarrow N$ be an immersion of smooth manifolds, in particular its tangent map $j_*: TM \rightarrow TN$ defines a vb-immersion over $j$ (cf.~definition~\ref{def.vbimmersion}).  
In such situation, there are two pertinent double vector bundles: first, the tangent bundle $T\nuup(j)$ of the normal bundle $\nuup(j)$ (left-side below), and secondly the vb-normal bundle $\nuup(j_*)$ of the vb-immersion $(j_*,j)$ (on the right-side):
 $$
\xymatrix@C=3pc@R=1.8pc{
T\nuup(j) 
\ar[d]_-{\pi^{\nuup(j)}_*} 
\ar[r]^-{\pi^{T\nuup(j)}_{\nuup(j)}}
&
\nuup(j) \ar[d]^-{\pi^{\nuup(j)}}
\\
 TM \ar[r]_-{\pi^{TM}} & M
}
\qquad \qquad 
\xymatrix@C=4pc{
\nuup(j_*) 
\ar[d]_-{\nuup^{(2)}(\underline{\pi})} 
\ar[r]^-{\pi^{\nuup(j_*)}_{TM}}
&
TM \ar[d]^-{\pi^{TM}}
\\
\nuup(j) \ar[r]_-{\pi^{\nuup(j)}}  & M
}
$$
The next result is a slight generalization of~\cite[Appendix A]{bursztyn-lima-meinrenken20019splitting}, the proof given here arises as a rather direct consequence of the formalism developed so far. 
\begin{prop}\label{prop-TN=NT}
There is a flip isomorphism
$\Upsilon^j: T \nuup(j) \longflipar \nuup(j_*)$. 
We will denote by $\Upsilon_j$ the inverse flip isomorphism: $\Upsilon^j=\Upsilon^{-1}_j$.
\end{prop}
\begin{proof}
    The expected flip map is obtained from the quotient flip $\Phi/\id$:  
    $$
    \xymatrix@R=1.5pc@C=2.5pc{
        0 \ar[r] 
        & TTM \ar[r]^-{j_{*\sharp}} \flipard_-{\id} 
        & (j_*)^*TTN \ar[r]^-{\quot_h} \flipard^-{\Phi}
        & (j_*)^*TTN \big/_{\!\! \scriptstyle h}\, TTM 
        \flipard^-{\Phi / \id} \ar[r] 
        & 0
        \\
        0 \ar[r]
        & TTM \ar[r]_-{j_{\sharp *}}
        & T(j^*TN) \ar[r]_-{\quot_v}
        & \frac{T(j^*TN)}{TTM}{\! \scriptstyle v}
        \ar[r] & 0
    }
    $$
    together with the dvb-isomorphism $ T\nuup(j)\cong \frac{T(j^*TN)}{TTM}{\! \scriptstyle v}$ provided by proposition~\ref{prop.quot-tgt}.
\end{proof}
\begin{remark}
    According to example~\ref{Phi-as-pb}, the flip isomorphism $\Phi/\id$ above  identifies canonically with $(j_*)^{*,\htov}\id_{TTN} / \id_{TTM}$.
\end{remark}
%
%
%

\section{The symmetry theorem}\label{sec.symm-thm}

\subsection{Immersions squares.}
Consider a commutative square consisting only of immersions of smooth manifolds, in other words a double morphism in $\Immer^{\square}$, 
\begin{align}\label{2-immersion}
\begin{split}
  \xymatrix@R=1.5pc@C=1.5pc{
    M_1 \xloopr[0.8pc]^-{i_1} 
    \ar@{}[rd]|-{(I,J)}
    %
    %
    \xloopd[0.8pc]_-{j_1}
    & 
    M_2
    %
    \xloopd[0.8pc]^-{j_2}
    \\
    N_1 
    \xloopr[0.8pc]_-{i_2}  
    & N_2
} 
\end{split}
\end{align}
Recall the two interpretations of the previous square as a 2-immersion of immersions:  vertically, $J: i_1 \Rightarrow i_2$, and horizontally $I: j_1 \Rightarrow j_2$.
Such diagram will be called \textbf{immersion square} and denoted by $(I,J)$.
\begin{prop}\label{prop-2immer}
    Let $(I,J)$ be an immersion square as in~(\ref{2-immersion}). 
    The following assertions are equivalent:
    \begin{enumerate}
        \setlength\itemsep{0pc}
        \item[$(1)$] The normal differential $\nuup^{(2)}(J): \nuup(i_1) \rightarrow \nuup(i_2)$ is a vb-immersion.
        \item[$(2)$] The normal differential $\nuup^{(2)}(I): \nuup(j_1) \rightarrow \nuup(j_2)$ is a vb-immersion.
        \item[$(3)$] For all $m\in M_1$, the following sequence of vector spaces is exact: 
        \begin{align*}
        \xymatrix@C=2.5pc{
            0 \ar[r] & T_mM_1 
            \ar[r]^-{\dtimes{j_{1*}}{i_{1*}}}
            &
            T_{j_1(m)}N_1 \times T_{i_1(m)}M_2 
            \ar[r]^-{i_{2*} - j_{2*}}
            &
            T_{n}N_2. 
        }
        \end{align*}
        where $n= i_2\circ j_1(m) = j_2\circ i_1(m)$.
    \end{enumerate}
\end{prop}
\begin{deff}
 An immersion square $(I,J)$ is called \textbf{regular} if it satisfies the conditions of proposition~\ref{prop-2immer}.
\end{deff}

\begin{remark} 
   The condition (3) reminds the "good pair condition" for the existence of fiber product of~\cite[appendix A]{bursztyn-cabrera-delhoyo-2016}. In particular, when $N_1$ and $M_2$ are embedded submanifolds of $N_2$ with clean intersection $M_1 = M_2 \cap N_1$, the condition (3) holds and $M_1$ is  the fiber product $M_2 \times_{N_2} N_1$. 
    In that clean intersection situation, the normal differential is a vb-embedding~\cite[Lemma 3.2.1]{pikethesis}. In terms of vector bundle, the exact sequence of condition (3) above reads as
    \begin{align*}
        \xymatrix@C=3pc{
            0 \ar[r] & TM_1 
            \ar[r]^-{\dtimes{j_{1*}}{i_{1*}}}
            &
            j_1^*TN_1 \oplus i_1^*TM_2 
            \ar[r]^-{i_{2*} - j_{2*}}
            &
            (I,J)^*TN_2
        }
    \end{align*}
        where $(I,J)^*TN_2$ refers to the pullback vector bundle $(j_2 \circ i_1)^*TN_2 = (i_2 \circ j_1)^*TN_2$.
\end{remark}
\begin{proof}[Proof of proposition~\ref{prop-2immer}.]
    By symmetry, we just need to show $(1)\Leftrightarrow (3)$.
    \par
    $(1) \Rightarrow (3)$: 
    Let $\delta \in T_{i_1(m)} M_2$ such that 
    $
    j_{2*}(\delta)$ belongs to $ i_{2*}(T_{j_1(m)} N_1)$, or equivalently $[\delta] \in \ker \nuup^{(2)}(J)$.
    Since $\nuup^{(2)}(J)$ is a vb-immersion, one must have $\delta = i_{1*}(\varepsilon)$ for some unique $\varepsilon$ in $T_m M_1$. 
    Put $\gamma= j_{1*}(\varepsilon)$ in $T_{j_1(m)} N_1$, hence $j_{2*}(\delta) = i_{2*}(\gamma)$.
    In particular, $(\delta,\gamma)$ lies inside $\range \dtimes{j_{1*}}{i_{1*}}$.
    \par
    $(3) \Rightarrow (1)$:
    Let $\delta \in T_{i_1(m)} M_2$ such that $[\delta] \in \ker \nuup^{(2)}(J)$, then  $j_{2*}(\delta) \in i_{2*}(TN_1)$. Let $\gamma \in T_{j_1(m)}N_1$ such that $j_{2*}(\delta)=i_{2*}(\gamma)$. By exactness, $(\gamma,\delta)$ belongs to $\range\dtimes{j_{1*}}{i_{1*}}$, thus $[\delta]=0$ in $\ker \nuup^{(2)}(J)$.
\end{proof}
\begin{remark}
    The inclusion $\range\dtimes{j_{1*}}{i_{1*}} \subseteq \ker(i_{2*} - j_{2*})$ always holds (direct consequence of the commutativity of the square~(\ref{2-immersion})), but the reverse inclusion is not automatic. For example, consider the following situation:
    $$
    \xymatrix@R=1.4pc@C=1.4pc{
    \R \xhookr[1pc] \xhookd[1pc] & \R^3 \xhookd[1pc]
    \\
    \R^5 \xhookr[1pc] & \R^6
    }
    $$
    where all the maps are the inclusions of the corresponding first coordinates. Then, applying the normal functor, say, on the horizontal 2-map returns the following vb-map
    $$
    \nuup(\R^5, \R^4) = \R^4 \times \R \rightarrow 
    \nuup(\R^6, \R^3) = \R^3 \times \R^3
    $$
    which cannot be a vb-immersion due to a dimension issue.
\end{remark}
As an application of the previous proposition, we obtain
\begin{cor}\label{cor-monosquare-specialcase}
    The following immersion square is regular:
    \begin{align}
    \begin{split}
  \xymatrix@R=1.4pc@C=1.4pc{
    M \xhookr[0.9pc]^-{j} 
    %
    %
    \xhookd[0.9pc]_-{\id\:}
    & 
    N
    %
    \xhookd[0.9pc]^-{\: i}
    \\
    M 
    \xhookr[0.9pc]_-{i \circ j}  
    & P
    } 
    \end{split}
    \end{align}
     In particular $\nuup_\VB \circ \nuup^{(2)}(j \circ i, j)$ is a well-defined double vector bundle.
\end{cor}
\begin{remark}
    It seems more tedious to prove directly that the induced map $\nuup(j) \rightarrow \nuup(i \circ j)$ is a vb-immersion. In order to do this, one may use the identification $\nuup(i \circ j) \cong j^*\nuup(i) \oplus \nuup(j)$ and attempt to match this normal differential with the inclusion of the $\nuup(j)$-component.
\end{remark}
%
%
%


\subsection{Statement and proof of the main theorem}
The time of the harvest has come.
In the present section, we prove our main result using the formalism developed so far.
\begin{theorem}[Symmetry theorem]\label{thm.main}
    Let $(I,J)$ be a regular immersion square, then there exists a canonical flip isomorphism
    $$
    \nuup_\VB \circ \nuup^{(2)}(J) \flipisomvar \nuup_\VB \circ \nuup^{(2)}(I).
    $$
\end{theorem}
%
As mentionned in the introduction, the latter flip map will be obtained by a double quotient procedure. More precisely, our approach relies on interpreting both double normal bundles $\nuup \circ \nuup^{(2)}(J)$ and $\nuup \circ \nuup^{(2)}(J)$ as double quotient themselves, arising from particular $3\times 3$ exact diagrams. Then, a flip isomorphism is obtained for each (but one) node of the latter diagram, some of which have already been discussed. Finally, their mutual compatibilities allows to consider the wanted double quotient. 
\begin{lem}\label{lem-pb}
Let $(I,J)$ be an immersion square as in~(\ref{2-immersion}), then there is a canonical dvb-isomorphism
\begin{align}\label{eq.lem-pb}
    (I^*j_{2*})^{*,h} (i_{2*})^{*,v} TTN_2 \cong 
    (J^*i_{2*})^{*,v} (j_{2*})^{*,h} TTN_2.
\end{align}
\end{lem}
\begin{proof}
 The swap isomorphism $TM_2 \times TN_1 \cong TN_1 \times TM_2$ induces a diffeomorphism 
 $$ 
 \begin{array}{rcl}
    (I^*j_{2*})^{*,h} (i_{2*})^{*,v} TTN_2
    &=&
    \left(
        M_1 \underset{i_1,\pi}{\times} TM_2
    \right)
    \underset{j_1 \times_{j_2} j_{2*}, \pi \times \pi^h}{\times} 
    \left(
       TN_1 \underset{i_{2*},\pi^v}{\times} TTN_2
    \right)
    \\
    &\cong& 
    \left(
        M_1 \underset{j_1,\pi}{\times} TN_1
    \right)
    \underset{i_1 \times_{i_2} i_{2*}, \pi \times \pi^v}{\times} 
    \left(
       TM_2 \underset{j_{2*},\pi^h}{\times} TTN_2
    \right)
    \\
    &=&
    (J^*i_{2*})^{*,v} (j_{2*})^{*,h}TTN_2.
 \end{array}
 $$
By a direct comparison of the respective homogeneity structures, we conclude that the dvb-structures coincides. 
\end{proof}
%
%
%
%
%
%
\begin{lem}\label{lem.Pslip}
There is a flip isomorphism 
\begin{align}\label{Pslip}
  \Psi = \Psi_{j_2}: (I^* j_{2*})^* T(i_2^* TN_2) \flipisomvar (J^*i_{2*})^* T(j_2^* TN_2). 
\end{align}
\end{lem}
\begin{proof}
 The desired flip map is induced by 
$$
TM_2^h \times TN_1^v \xrightarrow[\cong]{\flip} TM_2^v \times TN_1^h \xrightarrow[\cong]{\mathrm{swap}} TN_1^h \times TM_2^v 
$$
in the sense that $\Psi$ is the restriction-corestriction of the latter map, namely $\Psi$ makes the following square of double vector bundle commutative:
$$
\xymatrix@C=5pc@R=1.5pc{
   TM_1 \times (TM_2^h \times TN_1^v) \times TTN_2 \fliparr^-{\id \,\times (\mathrm{swap} \,\circ\, \flip ) \times \,\id} 
   & TM_1 \times (TN_1^v \times TM_2^h) \times TTN_2 
   \\
   (I^* j_{2*})^* T(i_2^* TN_2) \fliparr^-{\Psi} \xhooku[1pc] 
   & (J^*i_{2*})^* T(j_2^* TN_2) \xhooku[1pc]
}
$$
\end{proof}
\begin{remark}
 Observe that $\Psi_{j_2} \cong (J^*i_{2*})^{*, \vtoh} \Phi_{j_2}$, and that the lemma~\ref{lem-pb} together with lemma~\ref{lem.Pslip} provide the following  identification of flip isomorphisms:
 $$
 \xymatrix@C=8pc@R=1pc{
    (J^*i_{2*})^{*,v} (j_{2*})^{*,h}TTN_2 
    \fliparr^-{(J^*i_{2*})^{*, \varvtoh}(j_{2*})^{*,\varhtov} \id} 
    \ar@{}[d]|-{\rotatebox[origin=c]{90}{$\cong$}}
    & (J^*i_{2*})^{*,h} (j_{2*})^{*,v}TTN_2
      \ar@{}[d]|-{\rotatebox[origin=c]{90}{$\cong$}}
    \\
    (I^*j_{2*})^{*,h} (i_{2*})^{*,v}TTN_2 
    \fliparr_-{(I^*j_{2*})^{*, \varhtov}(i_{2*})^{*,\varvtoh} \id} 
    & (I^*j_{2*})^{*,v} (i_{2*})^{*,h}TTN_2
 }
 $$
\end{remark}
\begin{deff}
 Set $\Psi_{i_2}$ to be the flip map 
 $$(J^*i_{2*})^*T(j_2^*TN_2) \flipar (I^*j_{2*})^*T(i_2^*TN_2)$$ induced by the pullback flip $(I^*j_{2*})^{*, \varvtoh} \Phi_{i_2}$, and denote by $\Psi^{i_2} = \Psi^{-1}_{i_2}$ its inverse. 
\end{deff}
From the previous remark, $\Psi^{i_2} = \Psi_{j_2}$.
%
%
%
\begin{lem}\label{lem2x2}
    Let $(I,J)$ be a regular immersion square as in~(\ref{2-immersion}), then there is a canonical dvb-isomorphism
    $$
    \nuup(I^*j_{2*}) \cong \nuup^{(2)}(I)^{*,v} \nuup(j_{2*}).
    $$
\end{lem}
\begin{proof}
 By definition, 
 $$
    \nuup(I)^{*,v} \nuup(j_{2*}) 
    = \left(
    \frac{J^*i_{2*}}{i_{1*}}
    \right)^{*,v}
    \left(
    \frac{(j_{2*})^{*,h} TTN_2}{TTN_1}
    \!{\scriptstyle{h}} 
    \right)
 $$
 $$
  \nuup(I^* j_{2*}) 
    = \frac{(I^*j_{2*})^{*,h}(i_{2*})^{*,v} TTN_2}{(i_{1*})^{*,v}TTN_1}\!{\scriptstyle{h}}\, 
 $$
 Thus, the desired isomorphism follows from the compatibility of the vertical side-pullback  with the horizontal quotient of double vector bundle (proposition~\ref{prop.vpb-hquot}), together with the identification of  lemma~\ref{lem-pb}.
\end{proof}

%
%
%
Recall that the sharp differentials (example~\ref{ex.unbalanced}) of an immersion square $(I,J)$ are given by the following commutative square of vb-immersions:
\begin{align}
 \begin{split}
  \xymatrix{
    TM_1 
    \ar[r]^-{i_{1\sharp}} 
    \ar[d]_-{j_{1*}}
    \ar@{}[rd]|-{(I_\sharp, J_\flat)}
    & i_1^*TM_2 
    \ar[d]^-{I^* j_{2*}} 
    \\
    TN_1 \ar[r]_-{i_{2\sharp}} 
    & i_2^* TN_2
    }
    \qquad \qquad 
    \xymatrix{
    TM_1 
    \ar[r]^-{i_{1*}} 
    \ar[d]_-{j_{1\sharp}}
    \ar@{}[rd]|-{(I_\flat, J_\sharp)}
    & TM_2 
    \ar[d]^-{j_{2\sharp}} 
    \\
    j^*_1TN_1 \ar[r]_-{J^*i_{2*}} 
    & j_2^* TN_2
    }
 \end{split}
\end{align}
%
%
Moreover, the horizontal sharpening $(I^* j_{2*})_{\sharp,h}$ is, up to natural identifications, the horizontal side-pullback of the differential $(I^*j_{2*})_* : T(i_1^*TM_2) \rightarrow T(i_2^*TN_2)$ along the 2-map $\underline{\smash I^* j_{2*}} : \id_{i_1^*TM_2} \Rightarrow I^* j_{2*}$.
\begin{lem}\label{lem.dvbmap0}
There are canonical identifications:
\begin{itemize}
\setlength\itemsep{0pc}
 \item[(a)] $(I^* j_{2*})_* \cong (I_*)^{*,v}j_{2**}$
 \item[(b)] $(I^* j_{2*})_{\sharp,h} \cong (I_\sharp)^{*,v} j_{2*\sharp}$
\end{itemize}
\end{lem}
\begin{proof}
\textit{(a)} follows from the definition and
 \textit{(b)} is a direct consequence of lemma~\ref{lem-pb}.
\end{proof}
%
%
%
\begin{lem}\label{lem.dvbmap}
Consider the map $\Phi^{i_1}=\Phi_{i_1}^{-1}$ as in~(\ref{Phlip}) and the map $\Psi_{j_2}=\Psi^{i_2}$ from lemma~\ref{lem.Pslip}.
Then, we have the following identity of dvb-maps:
 $$\Psi^{i_2} \circ (I^* j_{2*})_\sharp  = I_\flat^* j_{2\sharp *} \circ \Phi^{i_1}.$$
\end{lem}
\begin{proof}
The expected identity is induced by the (obvious) commutativity of the square
$$
\xymatrix@C=11pc{
    TM_1^v \times TTM_2 
    \ar[r]^-{  
        (\id \, \times\, \mathrm{swap} \,\times \,\id)\, \circ\, \left( \dtimes{\underline{\smash \pi}}{j_{1*}^v}
        \times \dtimes{\underline{\smash \pi}^h}{j_{2**}}\right) 
        } 
    & M_1 \times TM_2^h \times TN_1 ^v \times TTN_2 
    \\
    TM_1^h \times TTM_2 
    \ar[r]_-{ \dtimes{\underline{\smash \pi}}{ j_{1*}^h} \times \dtimes{\underline{\smash \pi}^v}{j_{2**}} } 
    \fliparu^-{\flip \, \times\, \id}
    & 
    M_1 \times TN_1^h \times TM_2 ^v \times TTN_2
    \fliparu_-{\id\, \times (\mathrm{swap}\, \circ \, \flip) \times \,\id}
    }
$$
where $\underline{\smash \pi}$ denotes the obvious projections $TM_1^h \rightarrow M_1$ and $ TM_1^v \rightarrow M_1$, and $j_{1*}^v, j_{1*}^h$ are respectively the vertical and horizontal dvb-lift of the vb-map $j_{1*}$.
\end{proof}
%
%
%
%
%
%
\begin{lem}\label{lem.flip3x2}
 The quotient flip $\Psi_{i_2}/\Phi_{i_1}$ identifies canonically with a flip isomorphism 
 $$
  \Theta_{I}: \nuup^{(2)}(I)^*T\nuup(j_2)\flipisomvar \nuup(I^*j_{2*}).
 $$
 In the same manner, $\Psi_{j_2}/ \Phi_{j_1}$ identifies canonically to a flip isomorphism
 $$
 \Theta_J :  \nuup^{(2)}(J)^*T\nuup(i_2)\flipisomvar \nuup(J^*i_{2*}).
 $$
\end{lem}
\begin{proof} 
 Since both constructions are essentially the same, we only present the first one.
 By the lemma~\ref{lem.dvbmap}, the following horizontal 2-map
 $$
 \xymatrix@R=1.4pc@C=4pc{
    (i_1)^*TTM_2 \flipard_-{\Phi_{i_1}}
    \ar[r]^-{{I_\flat}^*j_{2\sharp *}}
    \rRightarrow
    &
    (J^* i_{2*})^* T(j_2^* TN_2) \flipard^-{\Psi_{i_2}}
     \\ 
    T(i_1^* TM_2) 
    \ar[r]_-{(I^*j_{2*})_{\sharp}} 
    &
    (I^* j_{2*})^* T(i_2^* TN_2)
 }
 $$
 induces the quotient flip
 $$
 \Psi_{i_2} / \Phi_{i_1} : \frac{(J^* i_{2*})^* T(j_2^* TN_2)}{(i_{1*})^*TTM_2}\!{\scriptstyle v} \longflipar (I^* j_{2*})^* T(i_2^* TN_2)\Big/_{\!\!\! h} T(i_1^* TM_2).
 $$
 On one side, the proposition~\ref{prop.quot-tgt} together with the (analog of) proposition~\ref{prop.vpb-hquot} provide a non-flipping canonical dvb-isomorphism
 $$
    \frac{(J^* i_{2*})^* T(j_2^* TN_2)}{(i_1)^*TTM_2}\!{\scriptstyle v} 
    \cong \left(\frac{J^*i_{2*}}{i_{1*}}\right)^{*,h} T\left(\frac{j_2^* TN_2}{TM_2} \right) = \nuup^{(2)}(I)^{*,h} T \nuup(j_{2}).
 $$
 On the other side, using lemma~\ref{lem-pb}, once again proposition~\ref{prop.vpb-hquot}, and lemma~\ref{lem2x2}, there are non-flipping canonical dvb-isomorphisms
 $$
 \begin{array}{rcl}
   (I^* j_{2*})^* T(i_2^* TN_2)\Big/_{\!\!\! h} T(i_1^* TM_2) 
    &\cong& 
    (J^* i_{2*})^{*,v} (j_{2*})^{*,h} TTN_2\Big/_{\!\!\! h} (i_{1*})^{*,v}TTM_2
     \\[0.5pc]
     &\cong& \displaystyle\left(\frac{J^*i_{2*}}{i_{1*}}\right)^{*,v} \left( (j_{2*})^{*,h} TTN_2\Big/_{\!\!\! h} TTM_2 \right) 
     \\[0.5pc]
     &=& \nuup^{(2)}(I)^{*,v} \nuup(j_{2*})
     \\ 
     &\cong& \nuup(I^* j_{2*}).
 \end{array}
 $$
 In conclusion, $\Psi_{i_2} / \Phi_{i_1}$ induces a flip map $\nuup^{(2)}(I)^* T\nuup(j_2) \flipar \nuup(I^*j_{2*})$. 
\end{proof}
Set $\Theta^I:= \Theta_I^{-1}$ and $\Theta^J=\Theta_J^{-1}$ for the respective inverses. 

%
%
%
\begin{lem}\label{lem.sq-flip}
With the previous notations, $\Theta_J \circ \nuup^{(2)}(J)_\sharp  = \nuup^{(2)}(J_\sharp) \circ \Upsilon^{i_1}$.
\end{lem}
\begin{proof}
The following square of dvb-maps (vertically) and flip isomorphisms (horizontally)
\begin{align}\label{sq1}
 \begin{split}
    \xymatrix@C=1.5pc@R=2pc{
    T\nuup(i_1) \fliparrr^-{\Upsilon^{i_1}} \ar[d]_-{\nuup^{(2)}(J)_\sharp} 
    &&
    \nuup(i_{1*}) 
    \ar[d]^-{\nuup^{(2)}(J_\sharp)}
    \\
    \nuup^{(2)}(J)^{*} T\nuup(i_2) 
    \fliparrr_-{\Theta_J} && 
    \nuup(J^*i_{2*})  
    }
 \end{split}
\end{align}
is isomorphic to the square 
\begin{align}\label{sq2}
 \begin{split}
 \xymatrix@C=2pc@R=2.5pc{
 \frac{T(i_1^*TM_2)}{TM_1}{\!\scriptstyle v} \fliparrr^-{\Phi^{i_1}/\id} \ar[d]_-{\frac{(I^*j_{2*})_\sharp}{j_{1*\sharp}}{\scriptscriptstyle \! v}} 
 &&
 \frac{(i_{1*})^*TTM_2}{TTM_1}{\scriptscriptstyle \! h}
 \ar[d]^-{\frac{I_\flat^*j_{2\sharp *}}{j_{1\sharp *}}{\scriptscriptstyle \! h}}
 \\
 \frac{(I^*j_{2*})^*T(i_2^*TN_2)}{(j_{1*})^*TTN_1}{\! \scriptstyle v} 
 \fliparrr_-{\Psi_{j_2}/\Phi_{j_1}} 
 && 
 \frac{(J^*i_{2*})^*T(j_2^*TN_2)}{T(j_1^*TN_1)}{\scriptscriptstyle \! h}  
 }
 \end{split}
\end{align}
thanks to propositions~\ref{prop.quot-tgt},~\ref{prop-TN=NT} and~\ref{prop.quot-flip}.
In particular, the square~(\ref{sq1}) commutes if and only if the square~(\ref{sq2}) does.
Now, the commutativity of (\ref{sq2}) is shown directly, as follows (recall that $\Psi^{i_2}=\Psi_{j_2}$): 
$$
\begin{array}{rcll}
  \displaystyle\frac{I_\flat^{*,h} j_{2\sharp *}}{j_{1\sharp *}}{\! \scriptstyle h} \circ \frac{\Phi^{i_1}}{\id} 
  &=& 
  \displaystyle\frac{I_\flat^{*,h} j_{2\sharp *} \circ \Phi^{i_1}}{j_{1\sharp *}}
  &  \text{(proposition~\ref{prop.quot-flip})}
  \\[0.8pc]
  &=&
  \displaystyle\frac{\Psi_{j_2} \circ (I^*j_{2*})_\sharp}{\Phi_{j_1} \circ j_{1*\sharp}}
  &   \text{(lemma~\ref{lem.dvbmap} and proposition~\ref{prop.Phlip-map})}
  \\[0.8pc]
  &=&
  \displaystyle\frac{\Psi_{j_2}}{\Phi_{j_1}} \circ \frac{(I^*j_{2*})_\sharp}{j_{1*\sharp}}{\! \scriptstyle v} 
  &  \text{(proposition~\ref{prop.quot-flip}, transposed)}.
\end{array}
$$
\end{proof}
\begin{lem}[]\label{lem.3x3}
 Given a regular immersion square $(I,J)$, there is a canonical isomorphism 
  $\nuup \circ \nuup^{(2)} (J) \cong \coker\,\nuup^{(2)}(I_\sharp)$ of double vector bundles.
\end{lem}
\begin{proof}
By definition, on one side
$
\nuup \circ \nuup^{(2)}(J) = \coker 
\left(\frac{I^*{j_{2*}}}{j_{1*}}\right)_{\sharp,h}
$ and, on the other side
$
\coker\, \nuup^{(2)}(I_\sharp) =  \frac{\nuup (I^*{j_{2*}})}{\nuup(j_{1*})}{\! \scriptscriptstyle v}. 
$
Thus, the conclusion follows from the identification
$
\frac{(I^*j_{2*})_\sharp}{j_{1*\sharp}}{\scriptscriptstyle v} \cong \left(\frac{I^*j_{2*}}{j_{1*}} \right)_{\sharp,h}
$
together with the fact that two canonically 2-isomorphic vb-maps have canonically isomorphic cokernels. 
The needed identification is provided by the following sequence of canonical identifications
$$
\begin{array}{rcll}
 \dfrac{(I^*j_{2*})_\sharp}{j_{1*\sharp}}{\scriptstyle \! v}
&\cong& 
\dfrac{(\underline{I^*j_{2*}})^*(I^*j_{2*})_*}{(\underline{j_{1*}})^* j_{1**}}{\scriptstyle \! v}
&\text{(sharpening as pullback, remark~\ref{rem.sharp-as-pb})}
\\[4mm]
&\cong&
\left( 
  \dfrac{\:\underline{I^*j_{2*}}\:}{\underline{j_{1*}}} 
\right)^*
\dfrac{(I^*j_{2*})_*}{j_{1**}}{\scriptstyle \! v}
& \text{(proposition~\ref{prop.vpb-hquot}, transposed version)}
\\[4mm]
&\cong&
\underline{
        \left( \dfrac{I^*j_{2*}}{j_{1*}} \right)
    }^*
\dfrac{(I^*j_{2*})_*}{j_{1**}}{\scriptstyle \! v}
& \left({\begin{tabular}{l}
    \text{direct check: the quotient of }
   \\
   \text{range-lift is range-lift of quotient}
  \end{tabular}
  }\right)
\\[4.5mm]
&\cong& 
\underline{
        \left( \dfrac{I^*j_{2*}}{j_{1*}} \right)
    }^*
\left(\dfrac{I^*j_{2*}}{j_{1*}}\right)_*
& \text{(proposition~\ref{prop.quot-tgt})}
\\[4.5mm]
&=& \underline{\nuup^{(2)}(J)}^*\nuup^{(2)}(J)_* 
\\[3mm]
&\cong&
\nuup^{(2)}(J)_\sharp
\end{array}
$$
\end{proof}
%
%
%
%
Gathering our results obtained so far gives the following $3 \times 3$ exact diagram of double vector bundles, in the following sense: each square commutes, each row is a short exact sequence of vertical vb-objects,  each column is an exact sequence of horizontal vb-objects (the zero objects have been omitted).
The $q$'s (resp. $\tilde q, \hat q, \bar q$) denote the corresponding quotient maps (resp. quotient map up to a canonical isomorphism).
{\footnotesize
\begin{align*}
\begin{split}
\arraycolsep=2pt
\xymatrix@C=1.6pc{
    {\left[\begin{matrix}
      TTM_1  & TM_1  
      \\ TM_1 & M_1  
    \end{matrix}\right]}
    \ar[rr]^-{
        \left[\begin{smallmatrix}
        i_{1\sharp*}& i_{1\sharp}  \\ \id & \id 
        \end{smallmatrix}\right] }
    \ar[d]_-{
        \left[\begin{smallmatrix}
        j_{1*\sharp}& \id  \\ j_{1\sharp} & \id 
    \end{smallmatrix}\right]
        }
    &&
    {\left[\begin{matrix}
      T(i_1^*TM_2) & i_1^*TM_2 
      \\ TM_1 & M_1 
    \end{matrix}\right] }
    \ar[r]^-{ 
        \left[\begin{smallmatrix}
        \tilde q& q  \\ \id & \id 
        \end{smallmatrix}\right]
        }
    \ar[d]_-{\left[\begin{smallmatrix}
        (I^* j_{2*})_\sharp & \id  \\ j_{1\sharp} & \id 
    \end{smallmatrix}\right]}
    &
    {\left[\begin{matrix}
       T\nuup(i_1)  & \nuup(i_1) 
       \\ TM_1 & M_1
    \end{matrix} \right]} 
    \ar[d]^-{
        \left[\begin{smallmatrix}
        \nuup^{(2)}(J)_\sharp & \id  \\ j_{1\sharp} & \id 
    \end{smallmatrix}\right]
    }
    \\
    {\left[\begin{matrix}
     (j_{1*})^*TTN_1 & TM_1  
     \\ j_1^* TN_1  & M_1
    \end{matrix} \right]}
    \ar[rr]^-{ 
        \left[\begin{smallmatrix}
          (J_\flat)^* i_{2\sharp *} & i_{1\sharp} \\ \id & \id 
        \end{smallmatrix}\right]
        }
        \ar[d]_-{\left[\begin{smallmatrix}
            q& \id  \\ q & \id 
            \end{smallmatrix}\right]}
    &&
    {\left[\begin{matrix}
     (I^* j_{2*})^*T(i_2^*TN_2)  & i_1^*TM_2
     \\ j_1^* TN_1  & M_1
    \end{matrix} \right] }
    \ar[r]^-{ 
        \left[\begin{smallmatrix}
            \hat q& q  \\ \id & \id 
        \end{smallmatrix}\right]
        }
    \ar[d]^-{\left[\begin{smallmatrix}
            q& \id  \\ q & \id 
        \end{smallmatrix}
        \right]}
    &
    {\left[\begin{matrix}
      \nuup^{(2)}(J)^* T \nuup(i_2)& \nuup(i_1) 
      \\ j_1^* TN_1 & M_1 
    \end{matrix}  \right] }
    \ar[d]_-{\left[\begin{smallmatrix}
            q& \id  \\ q & \id 
        \end{smallmatrix}\right]}
    \\
    {\left[\begin{matrix}
      \nuup(j_{1*})  & TM_1
      \\ \nuup(j_1) & M_1
    \end{matrix}\right]}
    \ar[rr]_-{ 
        \left[\begin{smallmatrix}
            \nuup^{(2)}(I_\sharp) & i_{1\sharp}  \\ \id & \id 
        \end{smallmatrix}\right]}
    &&
    {\left[\begin{matrix}
     \nuup(I^* j_{2*}) & i_1^* TM_2 
    \\ \nuup(j_1) & M_1
    \end{matrix}\right]}
    \ar[r]_-{ 
        \left[\begin{smallmatrix}
            \bar{q} & q  \\ \id & \id 
        \end{smallmatrix}\right]}
    &
    {\left[\begin{matrix}
     \nuup\circ \nuup^{(2)}(J) & \nuup(i_1) \\ \nuup(j_1) & M_1
    \end{matrix}\right]} 
}
\end{split}
\end{align*}
} 
 %
\begin{proof}[Proof of theorem~\ref{thm.main}.]
 Consider the quotient flip isomorphism
 $$
 \Theta_J / \Upsilon^{i_1} : \nuup^{(2)}(J)^*T\nuup(i_2) \big\slash_{\!\! \scriptstyle h} T\nuup(i_1) \longflipar \frac{\nuup(J^*i_{2*})}{\nuup(i_{1*})}{ \scriptstyle v}
 $$
 Then, by the transposed version of lemma~\ref{lem.3x3}, $\Theta_J/\Upsilon^{i_1}$ is identified to the desired flip isomorphism 
 $$
 \Lambda: \nuup \circ \nuup^{(2)}(J) \longflipar \nuup \circ \nuup^{(2)}(I)
 $$
  throught the dvb-isomorphism $\nuup \circ \nuup^{(2)}(I) \cong \frac{\nuup(J^*i_{2*})}{\nuup(i_{1*})}{ \scriptstyle v}$ on the codomain.
  Alternatively, the flip isomorphism $\Lambda$ is also obtained from $\Theta^I/\Upsilon_{j_1}$ by using the isomorphism of lemma~\ref{lem.3x3} on the domain instead.
\end{proof}
%
%
\section{Normal functor for vb-groupoids immersions}\label{sec.normalgroupoid}
A \textbf{Lie groupoid} $G \rightrightarrows G_0$ consists in two smooth manifold $G, G_0$, two surmersions $r,s: G \rightarrow G_0$, an embedding $u: G_0 \hookrightarrow G$, a diffeomorphism $\iota: G \rightarrow G$, and a composition map $\mu: G\times_{r,s} G \rightarrow G$, in such a way that the tuple $(G,G_0, r,s, u , \mu)$ defines a category with only invertible morphisms (the inversion map being $\iota$). We use the following usual notations: $\mu(g_1, g_2) := g_1g_2$, $\iota(g) := g^{-1}$, $u(x) := 1_x$.
A (strict) map of Lie groupoid $f: G \rightarrow H$ is the data of a pair of smooth map $f: G \rightarrow H$ and $f_0: G_0 \rightarrow H_0$ which are compatible with the whole structure, in other words $(f,f_0)$ defines a functor from $G$ to $H$. 
The class of Lie groupoids equipped with their maps form a category denoted $\LieGrpd$. 
In the same vein as smooth manifolds, we obtain the notion of groupoid immersion, groupoid submersion, groupoid embedding, etc... by requiring that both $f$ and $f_0$ belong to the corresponding class of maps.
\par \hspace{1pc}
Following~\cite{bursztyn-cabrera-delhoyo-2016}, a \textbf{vb-groupoid} is a Lie groupoid $E\rightrightarrows E_0$ equipped with a $(\R,\cdot)$-action $\delta: \R \times E \rightarrow E$ which is regular, and multiplicative in the sense that it acts by groupoid maps: $\delta(\lambda, g_1g_2) = \delta(\lambda, g_1) \delta(\lambda,g_2)$.
In that case, the homogeneity structure on consists in a pair of smooth $(\R,\cdot)$-actions $(\delta, \delta_0)$ compatible with the groupoid structure. 
Moreover, the subset of fixed point $G= \delta(0,E)$ admits a natural Lie groupoid structure with base manifold $G_0 = \delta_0(0,E_0)$.
The class of vb-groupoids equipped with their maps forms an additive category $\VB\Grpd$.
As for vector bundles, the vb-groupoid maps coincide with the $\R$-equivariant groupoid maps.  Morever, pullbacks and quotients of multiplicative $(\R,\cdot)$-actions are again multiplicative, in particular
\begin{prop}
\begin{itemize}
\setlength\itemsep{0pc}
    \item[]
    \item[(i)] Let $E\rightrightarrows E_0$ be a vb-groupoid with base groupoid $G$, and let $f: H \rightarrow G$ be a groupoid map. Then, the pullback  $f^*E \rightrightarrows f_0^*E_0$ has a natural structure of vb-groupoid with base groupoid $H$.
    \item[(ii)] Let $E,F$ be vb-groupoids over the Lie groupoid $G$ together with a wide vb-groupoid embedding $j: F \rightarrow E$. Then, the quotient $E/F$ has a natural structure of vb-groupoid.
\end{itemize}
\end{prop}
\begin{example}[Normal vb-groupoid]
 Let $j: H \looparrowright G$ be an groupoid immersion. Then the normal bundle $\nuup(j)$ inherits a vb-groupoid structure $\nuup(j) \rightrightarrows \nuup(j_0)$ over the Lie groupoid $H \rightrightarrows H_0$.
\end{example}
\par \hspace{1pc}
Mimicking the definition of double vector bundles, a \textbf{double vb-groupoid} is given by a Lie groupoid $D$ equipped with a pair of commuting regular multiplicative $(\R, \cdot)$-actions $(\delta^h, \delta^v)$.
Consider two groupoid immersions $j_1: G_1 \looparrowright H_1$, $j_2:  G_2 \looparrowright H_2$ and a 2-map $F: j_1 \Rightarrow j_2$ consisting of a pair $(f_2,f_1)$ of groupoid maps. Then, $\nuup^{(2)}(F): \nuup(j_1) \rightarrow \nuup(j_2)$ defines a vb-groupoid map.
\par
When $F=I$ is groupoid immersion (that is, an immersion in $\LieGrpd$), the notion of regular immersion square (cf. proposition~\ref{prop-2immer}) carries on directly: a  immersion square $(I,J)$ in $\LieGrpd$ is regular if the ``total square'' $(I,J)$ and ``base square'' $(I_0,J_0)$ are regular as immersion squares of smooth manifolds. In fact, the regularity of the base square turns out to be automatic\footnote{This is easily shown using the exact sequence characterization of the regularity of an immersion square (see proposition~\ref{prop-2immer}).}, namely $(I,J)$ is a regular immersion square in $\LieGrpd$ if and only if $(I,J)$ is a regular immersion square in $\Smooth$.
Our favorite example of double vb-groupoid arises then as the iteration $\nuup \circ \nuup^{(2)}(I)$ of the normal functor on a given a regular immersion square of Lie groupoids. 
\par\hspace{1pc}
Notice that the proof of the symmetry theorem (theorem~\ref{thm.main}) in the category $\Smooth$ relies only on the various properties of particular quotient and pullback operations for double vector bundles. Since these latter properties are shown by purely diagrammatic means, and the latter operations preserve the multiplicativity of the occuring homogeneity actions, they also hold when double vector bundles are replaced by double vb-groupoids. Consequently, we obtain 
\begin{theorem}[Symmetry theorem for double vb-groupoids]\label{thm.main_grpd}
Consider a regular immersion square $(I,J)$ in $\LieGrpd$
$$
\xymatrix@R=1.8pc{
G_1 \xloopr[1pc]^{i_1} \xloopd[0.8pc]_-{j_1} 
\ar@{}[rd]|{(I,J)}
& G_2 \xloopd[0.8pc]^-{j_2}
\\
H_1 \xloopr[1pc]_-{i_2} & H_2 
}
$$
Then, there is a canonical flip isomorphism of double vb-groupoids 
$$
\nuup\circ \nuup^{(2)}(I)  \longflipar \nuup \circ \nuup^{(2)}(J).
$$
\end{theorem}
\begin{cor}
 The Lie groupoid $\nuup \circ \nuup^{(2)}(I)$ and $\nuup \circ \nuup^{(2)}(I)$ are canonically isomorphic.
\end{cor}



\appendix
\section{Double category}\label{app.dbcat}
A category $(C,C_0, r,s, i,\circ)$ will be denoted shortly as $C \rightrightarrows C_0$. Here $C$ refers to the class of morphisms and $C_0$ is the class of objects. The structure maps consists in the source and range maps $r,s: C \rightarrow C_0$, the identity map $i: C_0 \rightarrow C$, and the composition map $\circ : C\times_{r,s} C \rightarrow C$. We will often write $C\times_{C_0} C$ instead of $C\times_{r,s} C$.
The category of categories equipped with functors as morphisms will be denoted by $\Cat$.
Following~\cite{Ehresmann1963double}, 
a double category would be a 2-dimensional category comprising  two (possibly distinct) classes of 1-morphisms which share the same objects, as well as the same 2-cells , all of them organized in a compatible way. 
Precisely, 
\begin{deff}\label{def.dbcat}
    A (strict) \textbf{double category} $\Dscr = (D,H,V,B)$ consists in:
    \begin{itemize}
    \setlength\itemsep{-2pt}
    \item A class of objects $B$,
    \item A horizontal category $(H, B, r^H, s^H, i^H, \circ)$,
    \item A vertical category $(V, B, r^V, s^V, i^V, \bullet)$,
    \item A class of double morphisms $D$, together with two category structures 
    $(D, H, r^D_H, s^D_H, i^D_H, \tilde \bullet)$ and $(D, V, r^D_V, s^D_V, i^D_V, \tilde \circ)$,
    \end{itemize}
    subject to the following axioms:
    \begin{itemize}
    \setlength\itemsep{0pc}
    \item (Axiom (5) in~\cite{Ehresmann1963double}, ``elements of $D$ are squares'')
    $$
    \begin{array}{rclrcl}
        s^V s^D_V &=& s^H s^D_H,
        \quad & r^V r^D_V &=& r^H r^D_H, 
        \\[2mm]
        s^H r^D_H &=& r^H s^D_H,
        \quad & s^V r^D_V &=& r^H s^D_H.
    \end{array}
    $$
    \item (Functoriality) The following diagrams commute:
    $$
    \xymatrix@C=1pc@R=1.5pc{
    D\times_V D \ar[d]_-{r^D_H \times r^D_H} \ar[r]^-{\tilde\circ} & D \ar[d]^{r_H^D}
    \\
    H \times_B H   \ar[r]_-{\circ}& H
    } 
    \qquad 
    \xymatrix@C=1pc@R=1.5pc{
    D\times_V D \ar[d]_-{s^D_H \times s^D_H} \ar[r]^-{\tilde\circ} & D \ar[d]^{s_H^D}
    \\
    H \times_B H   \ar[r]_-{\circ}& H
    } 
    \qquad  
    \xymatrix@C=1pc@R=1.5pc{
    D\times_V D \ar[r]^-{\tilde\circ}  &   D
    \\
    H \times_B H \ar[r]_-{\circ}  \ar[u]^-{i^D_H \times i^D_H}  
    & \ar[u]_{i_H^D} H
    }
    $$
    and similarily when the role of $H$ and  $V$ are permuted. By abuse of notation, we sometimes suppress the tilde in the notation $\tilde \circ$ and $\tilde\bullet$ when the context is clear.
    \item (Interchange law\footnote{``Permutability axiom'' in~\cite{Ehresmann1963double}}) Let $d_\ell \in D$ for $\ell =1,2,3,4$.
    If $d_1 \circ d_2$, $d_3 \circ d_4$, $d_1 \bullet d_3$ and $d_2 \bullet d_4$ are defined then 
    $$
    (d_1 \circ d_2) \bullet (d_3 \circ d_4) = (d_1 \bullet d_3) \circ(d_2 \bullet d_4).
    $$
    \end{itemize}
    A double category is called \textbf{flat} if double morphisms are determined by their edge morphisms.
  \end{deff}
  \begin{remark}
  The modified axiom (3') from~\cite{Ehresmann1963double}, regarding the functoriality of the source and range map,  turns to be  a consequence of axiom (5). Namely, $(r^D_V, r^H)$ and $(s^D_V, s^H)$ defines functors from $D \rightrightarrows H$ to $V\rightrightarrows B$, and $(r^D_H, r^V)$ and $(s^D_H, s^V)$ defines functors from $D \rightrightarrows V$ to $H \rightrightarrows B$.
  \end{remark}
\subsubsection{Choosing a direction.}\label{par.choosing-direction}
According to the last theorem in~\cite{Ehresmann1963double}, given a double category $\Dscr = (D,H,V,B)$, there are two distinct interpretations of $D$ as ``morphisms of morphisms''. In order to discard such ambiguity, we introduce the following terminology:  
\begin{itemize}
 \setlength\itemsep{0pc} 
  \item[(i)] the \textbf{verticalization} of $\Dscr$ is the internal category in $\Cat$ with category of morphisms  
  $\xymatrix@C=1.2pc{
        D \times_{H} D 
        \ar[r]^-{\tilde\bullet}
        & D \dar[r]
        & H, 
    }$ 
    and category of objects 
    $\xymatrix@C=1.2pc{
        V \times_{B} V 
        \ar[r]^-{\bullet}
        & V \dar[r]
        & B, 
    }$
  \item[(ii)] 
   the \textbf{horizontalization} of $\Dscr$ is the internal category in $\Cat$ with category of morphisms  
  $\xymatrix@C=1.2pc{
        D \times_{V} D 
        \ar[r]^-{\tilde\circ}
        & D \dar[r]
        & V ,}$ 
    and category of objects 
    $\xymatrix@C=1.2pc{
        H \times_{B} H
        \ar[r]^-{\circ}
        & H \dar[r]
        & B.
    }$
\end{itemize}
We will use the terminology vertical/horizontal \textbf{2-morphism} in a double category $\Dscr$ to refer to a 1-morphism in the category of morphisms of one of the internal category described above, namely a 2-morphism is \textit{directed} (in opposition to a double morphism).
Nowadays, it is common to find the asymmetric definition of a  double category as internal categories in $\Cat$. This latter approach was popularized in~\cite{grandis-pare1999} to define a weaker notion of double category, in which one of the directions is strictly associative (the horizontal one in loc. cit.) while the other one is only weakly associative. 
In the present work, we only encounter strict (flat) double categories in which both directions are equally relevant. 
But that's not all: the very essence of the proof of our main result (theorem~\ref{thm.main}) rests on the existence of a flip operation carrying the horizontalization onto the verticalization, and vice versa. 
In order to highlight the read direction of a square, we decorate it by a ``$\Rightarrow$''. For instance, if $v_1,v_2 \in V$ and $h_1, h_2$ in $H$, the diagrams
$$
 \xymatrix{
     \ar[r] \ar[d]_-{v_1} 
     \rRightarrow 
    & \ar[d]^-{v_2} 
    \\
     \ar[r] & 
  }
 \qquad \quad 
 \xymatrix{
     \ar[r]^-{h_1} \ar[d]_-{v_1}
     \ar@{}[rd]|{(h,v)}
    & \ar[d]^-{v_2} 
    \\
    \ar[r]_-{h_2} & 
}
\qquad \quad 
  \xymatrix{
     \ar[r]^-{h_1} \ar[d] 
     \dRightarrow
    & \ar[d] 
    \\
    \ar[r]_-{h_2} & 
}
$$
represents respectively a 2-morphism $v_1 \Rightarrow v_2$ in the horizontalization (LHS), and a 2-morphism $h_1\Rightarrow h_2$ in the verticalization (RHS). 
The square representing a double morphism $d$ whose horizontalization is $h= (h_2,h_1)$ and verticalization in $v=(v_2,v_1)$ will be annotated with the pair $(h,v)$ (middle square). 
\subsubsection{Flip of a double category.}\label{par.dbcat-flip}
Let $\Dscr= (D, H,V, B)$ be a double category, then there is an associated double category\footnote{This has been called ``transpose'' of $\Dscr$ in the literature, but we prefer to use the word ``flip'' in order to stick to the double vector bundle terminology.} $\flip(\Dscr) = (D, V,H, B)$  in which the horizontal and vertical directions are formally swapped: 
$$  
\left.\vcenter{
\xymatrix@C=3.5pc{
    b_1 \ar[r] \ar[d] \ar@{}[rd]|-{d= (h,v)} 
    & b_2\ar[d] 
    \\
    b_3 \ar[r] & b_4
} 
}\right.
\quad 
\overset{\flip}{\longmapsto} 
\quad
\left.\vcenter{
\xymatrix@C=4.5pc{
    b_1 \ar[r] \ar[d] \ar@{}[rd]|-{\flip(d)= (v,h)} 
    & b_3\ar[d] 
    \\
    b_2 \ar[r] & b_4
} 
}\right.
$$
In particular, the horizontalization/verticalization of $\flip(\Dscr)$ coincides the verticalization/horizontalization of the $\Dscr$, respectively.
$$
\xymatrix{
\text{Double category} \ar[r] \ar[d] &  \text{Horizontalization} 
\\
\text{Verticalization} \hspace{1cm}\ar@{<->}[ru]_{\flip}&
}
$$
%
%
\section{Vertical functoriality of the normal differential}\label{app.vert-fun}
In this appendix, we study the functoriality of the normal differential with respect to the vertical composition (proposition~\ref{prop.vert-fun}).
Recall that, the normal differential $\nuup^{(2)}(F)$ of a horizontal 2-map of immersions $F= (f_2,f_1): j_1 \Rightarrow j_2$ was defined as the quotient vb-map $J^*f_{2*}/f_{1*}$ where $J$ is the verticalization of $F$ (definition~\ref{def-normaldiff}). 
In what follows, we exploit the fact that such quotient is a cokernel in some suitable category, namely $\nuup^{(2)}(F)=\coker(J_\sharp)$.
\par \hspace{1pc}
Let $K$ and $H$ be a pair of vertically composable 2-maps of immersions with associated 2-immersions $J$ and $I$ respectively (their verticalizations), see below for some more specific notations. Then, by applying the tangent functor $T$ we obtain the composite $J_* \circ I_*$ of 2-vb-immersions: 
$$
\vcenter{
 \xymatrix@R=2.4pc@C=3.5pc{
    M_2 \xloopd[1pc]_-{i_2} \ar@{}[rd]|{(H,I)}& \ar[l]_-{h_1} \xloopd[1pc]^-{i_1} M_1 
    \\
    N_2 \xloopd[1pc]_-{j_2} \ar@{}[rd]|{(K,J)} & \ar[l]^-{h_2 = k_1} N_1 \xloopd[1pc]^-{j_1}
    \\
    P_2 & \ar[l]^-{k_2} P_1
 }}
 \qquad \overset{T}{\longmapsto} \qquad 
 \vcenter{
 \xymatrix@R=2.4pc@C=3.5pc{
    TM_2 \xloopd[1pc]_-{i_{2*}} \ar@{}[rd]|{(H_*,I_*)}& \ar[l]_-{h_{1*}} \xloopd[1pc]^-{i_{1*}} TM_1 
    \\
    TN_2 \xloopd[1pc]_-{j_{2*}} \ar@{}[rd]|{(K_*,J_*)} & \ar[l]^-{h_{2*} = k_{1*}} TN_1 \xloopd[1pc]^-{j_{1*}}
    \\
    TP_2 & \ar[l]^-{k_{2*}} TP_1
 }}
 $$
 Then, by performing a sharpening square by square, from the bottom to  top, we get:
 $$
 \vcenter{
 \xymatrix@R=2.4pc{
    TM_2 \xloopd[1pc]_-{i_{2*}} \ar@{}[rd]|{(H_*,I_*)}& \ar[l]_-{h_{1*}} \xloopd[1pc]^-{i_{1*}} TM_1 
    \\
    TN_2 \xloopd[1pc]_-{j_{2\sharp}} \ar@{}[rd]|{(K_\flat,J_\sharp)} & \ar[l]^-{k_{1*}} TN_1 \xloopd[1pc]^-{j_{1\sharp}}
    \\
    j_2^*TP_2 & \ar[l]^-{J^*k_{2*}} j_1^*TP_1
 }}
 \qquad \quad 
 \vcenter{
 \xymatrix@R=2.4pc@C=3.5pc{
    TM_2 \xloopd[1pc]_-{i_{2\sharp}} \ar@{}[rd]|{(H_\flat,I_\sharp)}& \ar[l]_-{h_{1*}} \xloopd[1pc]^-{i_{1\sharp}} TM_1 
    \\
    i_2^*TN_2 \xloopd[1pc]_-{i_2^*j_{2\sharp}} \ar@{}[rd]|{(I^*K_\flat,\,\mathbb{I}^*J_\sharp)} & \ar[l]_-{I^*h_{2*}} i_1^*TN_1 \xloopd[1pc]^-{i_1^*j_{1\sharp}}
    \\
    i_2^*j_2^*TP_2 & \ar[l]^-{I^*J^*k_{2*}} i_1^*j_1^*TP_1
 }}
 $$
 where $\mathbb{I}$ is a ``3-map'' from $\id_{h_1}: h_1 \Rightarrow h_1$ to $\id_{h_2}: h_2 \Rightarrow h_2$, given by the following commutative cube  
 $$
    \xymatrix@R=0.7pc@C=1.2pc{
        &&  M_1 \ar'[d]^-{h_1}[dd] \ar[rrrr]|{\,i_1\,} & & &
        & N_1  \ar[dd]^{h_2} & & &
        \\
        &  M_1 \ar[ur]^-{\id} \ar[dd]_-{h_1} \ar[rrrr]|(.6){\,i_1\,}  & &  & &
        N_1 \ar[ru]_-{\id} \ar[dd]_(.3){h_2} & & & &
        \\
        & & M_2\ar'[rrr]|-{\,i_2\,}[rrrr] 
        &  & 
        & & N_2 &  &
        \\
        & M_2 \ar[ur]^-{\id}\ar[rrrr]|-{\,i_2\,}& &
        &  &  N_2\ar[ru]_-{\id} &  &
        }
$$
and, $\mathbb{I}^*J_\sharp$ refers to the pullback of the square of vb-maps presenting $J_\sharp$, which maps to $\id_{h_2}$ under the base-projection double functor $\pi:\VB^\square \rightarrow \Smooth^\square$, along the 3-map $\mathbb{I}$ (obtained from a direct/naive generalization of pulling back vb-maps along 2-maps). 
 
 In sum, 
 $$
 \coker((J\circ I)_\sharp) \cong \coker(\mathbb{I}^*J_\sharp \circ I_\sharp)
 $$
 and then, since $\mathbb{I}^*J_\sharp$ is mono, the kernel-cokernel sequence reduces to a short exact sequence
 $$
 \xymatrix@C=1.5pc{
 0 \ar[r] & \coker(I_\sharp) \ar[r] & \coker(\mathbb{I}^*J_\sharp \circ I_\sharp) \ar[r] & \coker(\mathbb{I}^*J_\sharp) \ar[r]& 0.
 }
 $$
 Finally, the conclusion follows from the isomorphism
 $\coker(\mathbb{I}^*J_\sharp) \cong I^*\coker(J_\sharp)$
 and, the fact that the short exact sequence above splits.
\printbibliography
\vspace{5mm}
\noindent
Quentin Karegar Baneh Kohal \\
Instituto de Matem\'aticas \\
Universidad Nacional Aut\'onoma de M\'exico \\
Universidad s/n, Colonia Lomas de Chamilpa  \\
CP62210 Cuernavaca, Morelos Mexico \\
quentin.karegar@im.unam.mx
\end{document}